\documentclass[12pt, reqno]{amsart}
\usepackage{amssymb,amsmath,dsfont}
\usepackage{pgfplots}
\pgfplotsset{compat=1.14}
\usepackage{enumitem}
\usepackage{calc}

\title{Infinite cycles in the interchange process in five dimensions}
\date{}
\author{Dor Elboim}
\address{Dor Elboim\hfill\break
    Department of Mathematics,
    Princeton University,
    New Jersey, United States.}
\email{delboim@princeton.edu}
\author{Allan Sly}
\address{Allan Sly\hfill\break
    Department of Mathematics,
    Princeton University,
    New Jersey, United States.}
\email{allansly@princeton.edu}

\usepackage[margin=1in]{geometry}
\usepackage[all]{xy}
\usepackage{amssymb}
\usepackage{amsfonts}
\usepackage{amsthm}
\usepackage{amsmath}
\usepackage{hyperref}
\usepackage{mathtools}
\usepackage{url}
\usepackage{xcolor}
\usepackage{dsfont}
\usepackage{pgfplots}
\usepackage{comment}
\usepackage{mathrsfs}

\newtheorem{mainthm}{Theorem}
\newtheorem{thm}{Theorem}[section]

\newtheorem{lem}[thm]{Lemma}  
\newtheorem{prop}[thm]{Proposition}
\newtheorem{cor}[thm]{Corollary}

\newtheorem{definition}[thm]{Definition}
 
\newtheorem*{conjecture}{Conjecture}
\newtheorem{remark}[thm]{Remark}
\newtheorem{claim}[thm]{Claim}

\renewcommand{\P}{\mathbb{P}}

\mathtoolsset{showonlyrefs}

\numberwithin{equation}{section}

\def\P{\mathbb{P}}
\def\E{\mathbb{E}}
\def\R{\mathbb{R}}
\def\Z{\mathbb{Z}}
\def\cA{\mathcal{A}}
\def\cB{\mathcal{B}}
\def\cC{\mathcal{C}}
\def\cD{\mathcal{D}}

\def\cF{\mathcal{F}}
\def\cG{\mathcal{G}}
\def\cH{\mathcal{H}}

\def\fP{\mathfrak{P}}
\def\fR{\mathfrak{R}}

\begin{document}

\maketitle

\begin{abstract}
In the interchange process on a graph $G=(V,E)$, distinguished particles are placed on the vertices of $G$ with independent Poisson clocks on the edges. When the clock of an edge rings, the two particles on the two sides of the edge interchange. In this way, a random permutation $\pi _\beta:V\to V$ is formed for any time $\beta >0$. One of the main objects of study is the cycle structure of the random permutation and the emergence of long cycles.

We prove the existence of infinite cycles in the interchange process on $\mathbb Z ^d$ for all dimensions $d\ge 5$ and all large $\beta $, establishing a conjecture of Bálint Tóth from 1993 in these dimensions.

In our proof, we study a self-interacting random walk called the cyclic time random walk.  Using a multiscale induction we prove that it is diffusive and can be coupled with Brownian motion. One of the key ideas in the proof is establishing a local escape property which shows that the walk will quickly escape when it is entangled in its history in complicated ways.
\end{abstract}

\section{Introduction}

The interchange process (sometimes also called the random stirring model) is a model of random permutations obtained by composing random transposition along the edges of a graph. Given a graph $G=(V,E)$ the model is defined as follows. On each vertex of the graph there is a distinguished particle and each edge of the graph is endowed with an independent Poisson clock of rate $1$. The process evolves according to the following simple rule: when an edge rings, the two particles on the two sides of the edge swap. In this way, a random permutation $\pi _\beta  :V\to V$ is formed for any time $\beta >0$. 

The study of the cycle structure of the permutation $\pi _\beta $ is of great interest. This was initiated by Bálint Tóth \cite{toth1993improved} (and independently by Aizenman and Nachtergaele \cite{aizenman1994geometric} and Powers \cite{powers1976heisenberg}) who observed the connection between long cycles in the interchange process on $\mathbb Z ^d$ and a phase transition in the quantum Heisenberg ferromagnet. Tóth's conjecture is as follows.

\begin{conjecture}[Tóth 1993]
    Let $\pi _\beta $ be the interchange permutation on $\mathbb Z ^d$ at time $\beta$. 
    \begin{enumerate}
        \item 
        When $d=2$, for all $\beta >0$, the permutation $\pi _\beta $ contains only finite cycles almost surely. 
        \item 
        When $d\ge 3$, for all $\beta >\beta _c$, the permutation $\pi _\beta $ contains infinite cycles almost surely.
    \end{enumerate}
\end{conjecture}
Since the work of Tóth, the cycle structure of $\pi _\beta $ has been studied extensively for various graphs. The existence of long cycles has been shown on the complete graph, trees, the hypercube, the Hamming graph and random regular graphs. However, proving similar results for $\mathbb Z ^d$ was out of reach. In this paper we prove the existence of infinite cycles in dimensions 5 and higher.
\begin{mainthm}\label{t:main.thm}
    For all $d\ge 5$ and all $\beta $ sufficiently large, the permutation $\pi _\beta =\pi _\beta (\mathbb Z ^d)$ contains infinite cycles almost surely.
\end{mainthm}

In the proof we study the cyclic random walk introduced in \cite{toth1993improved}. This is the walk obtained by exposing the cycle containing the origin in $\pi _\beta $. We use a multiscale inductive argument to show that this walk is diffusive. The method shows that for large $\beta $, the probability that the cycle containing the origin is finite is approximately the probability that a simple random walk will return to the origin at a cyclic time $k\beta $ with $k\in \mathbb N$. That is, we show that this probability is roughly $\beta ^{-d/2}$.

\subsection{Previous works}

\subsubsection{The cycle structure of the interchange}

It is of great interest to study the lengths of the cycles in the random permutation $\pi _\beta =\pi _\beta (G)$ for various graphs $G$ and in particular to prove the emergence of long cycles. This was done on the complete graph by Berestycki and Durrett~\cite{berestycki2006phase} and later refined by Schramm~\cite{schramm2005compositions} who proved the convergence of the cycle lengths to the Poisson Dirichlet distribution. The proof was later simplified and generalized by Berestycki \cite{berestycki2011emergence}. Berestycki and Kozma found yet another proof of the result using representation theory \cite{berestycki2012cycle}.

The existence of infinite cycles was shown on the infinite $d$ regular tree with $d\ge 5$ by Angel \cite{angel2003random}. This was later improved by Hammond \cite{hammond2012infinite} for a wider class of trees including the infinite binary tree. In a subsequent paper \cite{hammond2015sharp}, Hammond also proved the existence of a critical time for the $d$ regular tree when $d$ is sufficiently large. 
There are numerous other results showing the emergence of long cycles on the hypercube \cite{kotecky2016random,hermon2021interchange}, on Hamming graphs \cite{adamczak2021phase} and on random regular graphs \cite{poudevigne2022macroscopic}.
 Finally,  Alon and Kozma~\cite{alon2013probability} developed a formula for the probability of the (very rare) event that the permutation contains only one cycle in terms of the eigenvalues of the graph $G$. This result is proved using representation theory and it holds for all finite graphs. 

\subsubsection{Mixing time and spectral gap}
When $G$ is the complete graph, the interchange process is the random transposition shuffle of a deck of cards. At each step one chooses two cards uniformly at random from the deck and transposes them. In one of the founding results in the theory of Markov chain mixing, Diaconis and Shahshahani \cite{diaconis1981generating} showed that the deck will be mixed after $\frac{1}{2}n \log n$ random transpositions. This result is one of the first examples of the cutoff phenomenon in Markov chains and introduced the use of representation theory in the study of random walks on groups. 

Since the work of Diaconis and Shahshahani, the mixing time and spectral gap of the interchange process have been studied extensively. One notable result is the proof of Aldous conjecture by  Caputo, Liggett and Richthammer \cite{caputo2010proof}. The conjecture (now a theorem) says that, on any graph, the spectral gap of the interchange process is equal to the spectral gap of the simple random walk on the graph.
We refer the reader to the following results on the spectral gap \cite{flatto1985random,cesi2010eigenvalues,handjani1996rate,koma1997spectral,morris2008spectral,starr2011asymptotics,dieker2010interlacings} and the mixing time \cite{diaconis1988group,wilson2004mixing,aldous1995reversible,oliveira2013mixing,jonasson2012mixing,lee1998logarithmic,yau1997logarithmic,diaconis1993comparison,alon2020comparing,lacoin2016mixing,hermon2021interchange} of the interchange process on various graphs. See also \cite{subag2013lower,bernstein2019cutoff,nestoridi2021mixing,berestycki2011mixing} for similar mixing results on different variants of the process.

\subsection{Extensions and open problems}

\subsubsection{The finite volume case}

Consider the interchange process on the torus graph $\Lambda _L:=[1,\dots ,L)^d$. Using the methods of this paper, it is not hard to show that when $d\ge 5$, and $L\to \infty $, there are polynomially long cycles in $\pi _\beta (\Lambda _L)$ when $\beta $ is large. However, there are much more accurate predictions for the cycle lengths in this case. Tóth conjectured that macroscopic cycles are formed in the supercritical regime when $\beta >\beta _c$. These are cycles with lengths proportional to the volume of the torus $L^d$. Moreover, it is expected that the longest cycles will obey the so called Poisson-Dirichlet law. More precisely, for all $d\ge 3$ and all $\beta >\beta _c$ there exists a fixed fraction $0<\alpha =\alpha _\beta <1 $ such that roughly $\alpha L^d$ of the particles belong to a macroscopic cycle and roughly $(1-\alpha )L^d$ of the particle are in cycles of length $O(1)$ as $L\to \infty $. Letting $\ell _1 ,\ell _2, \dots $ be the lengths of the longest cycle in $\pi _\beta $, second longest cycle and so on, it is conjectured that
\begin{equation}
    \Big( \frac{\ell _1}{\alpha L^d}, \frac{\ell _2}{\alpha L^d}, \dots  \Big) \overset{d}{\longrightarrow } \text{PD}(1), \quad L\to \infty,
\end{equation}
where $\text{PD}(1)$ is the Poisson Dirichlet distribution with parameter $1$. See \cite{elboim2019limit} for more details and the exact definition of the Poisson Dirichlet distribution.

Such results were obtained by Schramm on the complete graph \cite{schramm2005compositions}. The proof of \cite{schramm2005compositions} contains two parts.  In the first part, it is shown that macroscopic cycles emerge at the right time. In the second part, the author uses a coupling argument to show that once long cycles are formed they quickly converge to the Poisson-Dirichlet distribution. It is possible that using a combination of the methods in this paper and in \cite{schramm2005compositions} will allow us to obtain the Poisson-Diriclet distribution in $\Lambda _n$. We intend to pursue this direction in the future.

\subsubsection{Other dimensions}
Our proof holds in dimension $d\ge 5$ because of the following heuristic. A random walk in five dimension will completely avoid its past with positive probability whereas in four dimension it will intersect its past with high probability. This fact allows us to couple the walk with shorter independent walks in order to analyze its behavior.

Dimension four is in some sense critical for the property of avoiding the past. This is similar to the fact the dimension two is critical for the property of transience and recurrence. Indeed, consider the probability $p_n$ that a simple random walk after time $n$ avoids its history before time $n$. That is, 
\begin{equation*}
    p_n:=\mathbb P \big( \exists s \le n-1,\  t\ge n \text{ such that }   W(s)=W(t) \big),
\end{equation*}
where $W$ is a simple $d$ dimensional random walk. It is well known that $p_n$ decays polynomially with $n$ when $d\le 3$, is bounded away from zero when $d\ge 5$ and decays like $1/\log n$ when $d=4$.

In the future we hope  to push our methods to dimension 4 using the fact that $p_n$ decays so slowly to zero in this case. To this end, one has to understand better the self intersections of the walk and argue that these intersections do not change the diffusive behavior of the walk too much. 

\subsubsection{The existence of a critical time}
A simple percolation argument shows that in any dimension $d\ge 2$ if $\beta $ is sufficiently small then there are no infinite cycles in $\pi _\beta $. Indeed, consider the set of edges that rang at least once up to time $\beta $. This is a percolation with probability $1-e^{-\beta }$ and therefore if $1-e^{-\beta }<p_c(\mathbb Z ^d )$ then the resulting graph contains only finite connected components. It is clear that in this case $\pi _ \beta $ contains only finite cycles.

We note that, unlike percolation for example, there is no simple monotonicity in this model. It might be the case that infinite cycles appear at some time and then disappear at a later time. Thus, it is still open to prove the existence of a unique phase transition for all $d\ge 3$. That is, to prove that there exists a critical time $\beta _c$ such that for $\beta <\beta _c$ there are only finite cycles and for $\beta >\beta _c$ there are infinite cycles.

It is possible that using a combination of our methods and those of \cite{hammond2015sharp}, one can prove the existence of a phase transition when the dimension $d$ is sufficiently large.

\subsubsection{Weakly correlated walks}
Many self-interacting random walks, such as the self avoiding walk or weakly self avoiding walk have been successfully studied using the Lace Expansion.  Our inductive method seems to be quite robust in the sense that it does not use heavily the exact local interactions of the interchange process. It is natural to ask if our proof extends to any of these walks.  

\subsection{Related models}

\subsubsection{Spatial random permutations}
The interchange process on $\mathbb Z ^d$ can be thought of as part of a wider family of models of spatial random permutations. In these models a random permutation is sampled such that $x$ and $\pi (x)$ are typically close in some underlying geometry.
Two examples of such models are the Mallows measure defined in one dimension and the Euclidean random permutations model. The second model was introduces by Matsubara \cite{matsubara1951quantum} and Feynman \cite{feynman1953atomic} to study the phenomenon of Bose-Einstein condensation in the ideal Bose gas. The model is defined as follows. For a given dimension $d$, time $t>0$ and length $L>0$, one samples $L^d$ particles in the continuous $d$-dimensional torus of side length $L$, and a permutation $\pi $ of the particles, with probability density proportional to 
\begin{equation}\label{eq:835}
    \prod _{i=1}^{L^d} \exp \big( |x_{\pi (i)}-x_i|^2/t \big)=\exp \bigg( \frac{1}{t}\sum _{i=1}^{L^d} |x_{\pi (i)}-x_i|^2 \bigg)
\end{equation}
where $x_i$ is the position of the $i$-th particle and where $|\cdot |$ denotes the Euclidean distance. Note that in here the particles and the permutation are sampled simultaneously with a joint density given by \eqref{eq:835}. The connection to the interchange in finite volume is clear. Indeed, in the interchange process any particle performs a simple continuous time walk and therefore its final position will be roughly normally distributed around its starting point with standard deviation $\sqrt{t}$. Thus, the density of the final position $\pi _t(x)$ of the particle starting at $x$ will be roughly given by the term inside the product in \eqref{eq:835}. Since different particles in the interchange process are somewhat independent, one may expect that a permutation sampled according to the product \eqref{eq:835} will share some similarities with the interchange.

The model defined in \eqref{eq:835} is in some sense integrable or exactly solvable and much more rigorous results are known. The analogue of Tóth conjecture was proved in \cite{betz2009spatial} and the Poisson Dirichlet limit law is established in \cite{betz2011spatial}.

In the finite volume interchange process in dimensions 2, it is conjectured that for all fixed $t>0$ there are only short and finite cycles in the permutation. However, when $t$ grows like $\log L$ (where $L$ is the side length of the box), it is expected that macroscopic cycles are starting to form. This prediction was rigorously established for the Euclidean permutations model by the first author and Peled \cite{elboim2019limit}. In this paper there are similar results in dimension $1$ and in the critical time in all dimensions. These results exceed the best known predictions for the interchange process and the possibility of universality is especially intriguing.

\subsubsection{The quantum Heisenberg ferromagnet}
Consider the following random permutation model on a graph $G=(V,E)$. For a permutation $\pi $ we let $C(\pi )$ denotes the number of cycles in $\pi $. Suppose that $\mu _\beta (\pi)$ is the probability to obtain the permutation $\pi :V \to V$ in the interchange process at time $\beta $. We define the new measure on permutation 
\begin{equation}\label{eq:1}
    \nu _t(\pi ):=\frac{1}{Z} \cdot 2^{C (\pi )} \cdot \mu _\beta (\pi ),
\end{equation}
where $Z$ is a normalizing constant. This is just the interchange measure on permutations tilted by an additional factor of $2$ to the power of the number of cycles. It is shown in \cite{toth1993improved} that this tilted model is in direct correspondence with the quantum Heisenberg model on $G$ at temperature $T=1/\beta $. More precisely, the spin-spin correlation between sites $x$ and $y$ in the quantum Heisenberg model equals a constant times the probability that $x$ and $y$ are in the same cycle. We denoted the time by $\beta $ throughout the paper sense it corresponds to the inverse temperature in the quantum Heisenberg ferromagnet.

A longstanding open problem is to rigorously prove the existence of a phase transition in the quantum Heisenberg ferromagnet on the torus $\Lambda _L:=[0,\dots ,L)^d$ in dimension $d\ge 3$. This is equivalent to showing the emergence of long cycles in the random permutation sampled according to \eqref{eq:1} with $G=\Lambda _L$. It is possible that using the methods developed in this paper, one can make progress toward this conjecture. 

Let us mention that in the antiferromagnetic case of the quantum Heisenberg model more is rigorously known. Indeed, in this case Dyson, Lieb and Simon \cite{dyson2004phase} proved the existence of long range order for all dimensions $d\ge 5$. This was later extended to dimensions $d\ge 3$ using further observations of Neves and Perez \cite{neves1986long}, and of Kennedy, Lieb and Shastry \cite{kennedy1988existence}. Finally, we refer the reader to the survey \cite{goldschmidt2011quantum} for further background on the quantum Heisenberg model.

\subsection{Organization of the paper}
In Section~\ref{sec:cyclic}, we define the cyclic walk and the regenerated cyclic walk and prove some basic a priori estimates for these walks. In Section~\ref{sec:induction} we state the induction hypothesis and in Section~\ref{sec:relaxed} we show that the walk will, most likely, avoid its history after a relaxed time. One of the main ideas of this paper is the escape algorithm which is presented in Section~\ref{sec:escape}. It shows that every time the walk enters a new small unexplored region, it has a non negligible chance of escaping far in a short period of time. The results of this section allows us to handle complicated situations in which the walk entangles with its history. In Section~\ref{sec:heavy} we prove that the walk does not spend too much time in any large block and that the the walk will quickly reach a relaxed time. These key estimates allows us to couple the long walk with a concatenation of shorter independent walks. In Section~\ref{sec:concatenation} we use this concatenation to prove the induction step. In Section~\ref{sec:base} we prove the induction base. Finally, in Section~\ref{sec:brownian} we prove some Brownian motion estimates that are used throughout the paper.

\subsection{Acknowledgments}
The first author would like to give a special thanks to Ron Peled for introducing to him the conjecture of Tóth in a seminar in Tel Aviv University around eight years ago and for many helpful discussions over the years. This conjecture is one of the reasons why D.E. decided to specialize in probability theory. D.E. thanks Ofir Gorodetsky, Gady Kozma, Yuval Peres and Daniel Ueltschi for fruitful discussions on the interchange process and related models. We thank Roland Bauerschmidt and Tom Spencer for meeting with us.  Indeed, their suggestions really helped get us started with the project. We thank Bálint Tóth for telling us about the history and background of the problem and of course for making the conjecture!  A.S. is supported by NSF grants DMS-1855527, DMS-1749103, a Simons Investigator grant, and a MacArthur Fellowship.

\section{Cyclic and regenerated walks}\label{sec:cyclic}

The continuous time \textbf{interchange process} $\pi_\beta$ is stochastic process taking values in the symmetric group on 
the vertex set of some graph $G=(V,E)$ with generator
\[
\mathscr{L}f(\pi) = \sum_{(u,v)\in E} (f((u v) \cdot \pi ) - f(\pi))
\]
It can be constructed as follows from rate 1 Poisson clocks $\Psi_e$ on each edge $e$ of the graph.  If edge $e=(u,v)$ rings at time $s$, the process is multiplied by the transposition $(u v)$ so $\pi_s=(uv) \cdot \pi_{s^-} $. 

For time $t=k\beta+s$ where $k\in\Z$ and $s\in[0,\beta)$ we define the \textbf{cyclic walk with period $\beta$} to be $W(t) := (\pi_s(\pi_\beta)^k)(0)$ which describes the position of the origin generated by the random transpositions which repeat cyclically with period $\beta$.  More concretely, the walk jumps from $u$ to $v$ at time $t=k\beta+s$ if $W(t-)=u$ and the edge $(u,v)$ rings at time $s$. Note that according to this definition the walk $W$ is defined for all $t\in \mathbb R$.

Up until time $\beta$ the transpositions are Poisson processes so the cyclic walk is simply a continuous time random walk in the interval $[0,\beta]$.  However, after time $\beta$ it may encounter jumps it has already taken and so there will be an interaction with its past.  In particular, the cyclic walk is not a Markov chain. A cyclic walk on $\Z$ is illustrated in Figure~\ref{fig:cyclic}.    

\begin{figure}[htp]
    \centering
    \includegraphics[width=15cm]{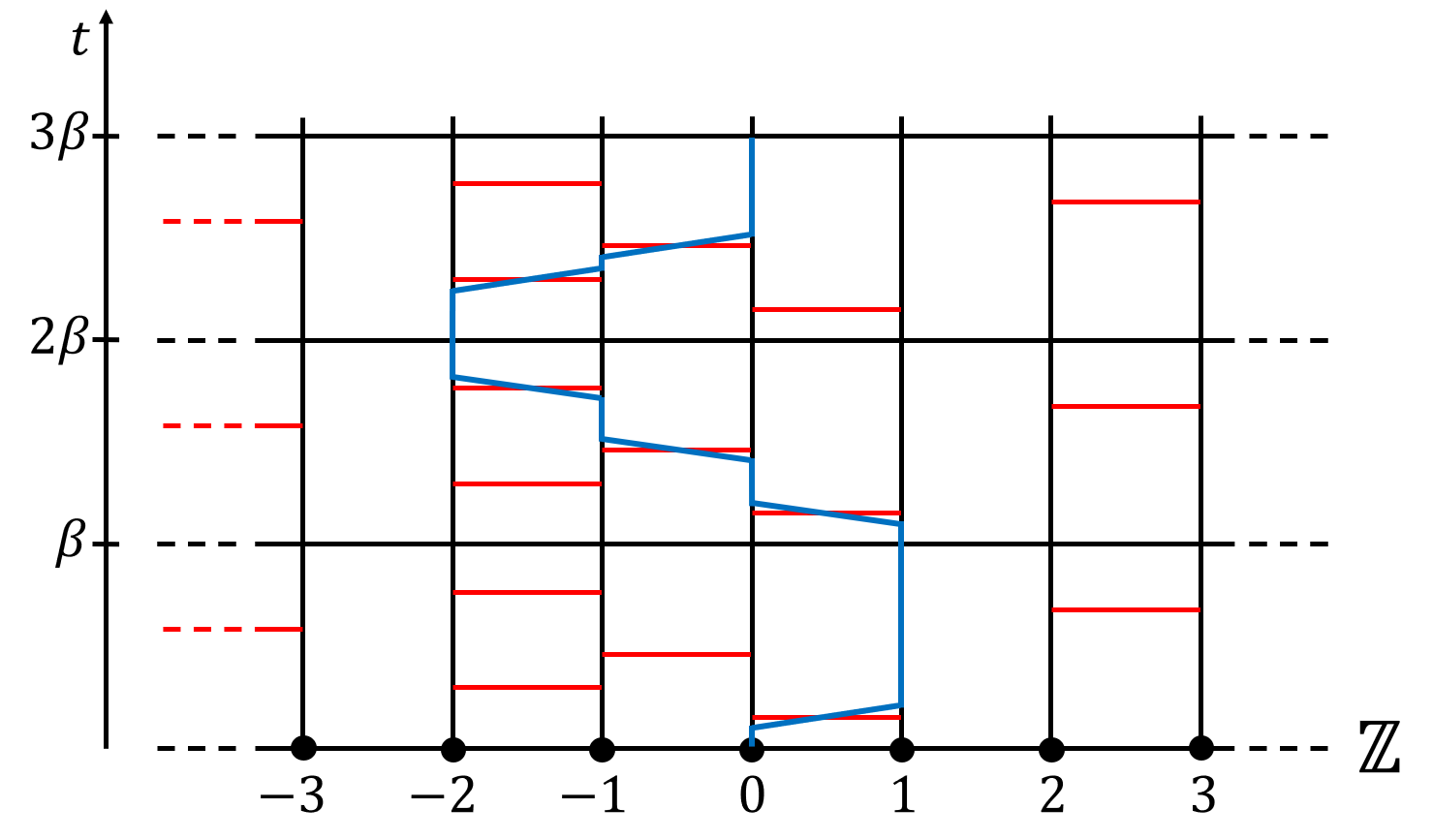}
    \caption{A cyclic time random walk on $\mathbb Z$. The red bars indicate a ring on the corresponding edge. The pattern repeats periodically when time is shifted up by $\beta $. The blue path is the cyclic time random walk $W(t)$ starting from the origin. When the walk meets a horizontal bar it jumps across it to the other side. In this example the walk closes at time $3\beta $ and therefore the cycle of the origin is of length $3$. Note also that $-1$ is a fixed point of $\pi _\beta $ and $(2,3)$ is a cycle of length $2$. }
    \label{fig:cyclic}
\end{figure}

From here let us assume that the graph is $\Z^d$ .
In order to sample the cyclic walk (say in a computer simulation), we do not need the randomness of the Poisson processes on all the edges. In fact, we only need the randomness of a simple, continuous time random walk. Let $X(t)$ be a simple continuous time walk, jumping with rate one to each one of its neighbors. We define $W(t)$ for $t>0$ in the following inductive way. Intuitively, $W$ takes exactly the same steps as $X$, unless it is forced to do something else by its history. That is, the jump is suppressed if $W$ is attempting to jump into its history and a jump is forced if its history jumps into it. More precisely, suppose that $W$ was defined for all $s\le t^-$. Then, 
\begin{enumerate}
    \item 
    If $X$ jumps at time $t$ and for all $k\in \mathbb N $ we have that 
    \begin{equation}
        W(t^-)+X(t)-X(t^-)\neq W(t-k\beta )
    \end{equation}
    then we take the jump: $W(t)-W(t^-)=X(t)-X(t^-)$. Otherwise, the jump is suppressed: $W(t)=W(t^-)$.
    \item 
    Similarly, if for some $k\in \mathbb N$ we have that $W$ jumps at time $t-k\beta $ and $W(t-k\beta )=W(t^-)$ then $W$ is forced to jump: $ W(t)=W((t-k\beta)^-)$.
\end{enumerate}
It is clear that the walk $W$ defined in the above way has the same law as the cyclic random walk. We say that the cyclic random walk $W(t)$ is  \textbf{driven}  by the simple random walk $X(t)$. We let $\mathcal F _t$ be the sigma algebra generated by the walk $X(s)$ for $s\in [0,t]$.

We say that the walk $W$ is \textbf{interacting with its past} at time $t>0$ if
\[
\inf \big\{ |W(t)-W(t-k\beta)|: k \in \mathbb N , \  k\le t/\beta  \big\} \leq 1.
\]
If $W(t)$ does not interact with its past in $[t,t']$ then $W(s)-W(t)=X(s)-X(t)$ for all $s\in[t,t']$.

If the origin is in a finite cycle then for some positive integer $k$ we have that $W(k\beta)=0$ and we define
\begin{equation}\label{eq:def of tau reg}
\tau _{\text{reg}}=\inf\{k\beta:k\in \mathbb N, W(k\beta)=0\}.    
\end{equation}
After $\tau _{\text{reg}}$, the cyclic walk repeats its past and so $W(s+\tau _{\text{reg}})=W(s)$ for all $s\geq 0$.  We call $\tau _{\text{reg}}$ the \textbf{time the cycle closes} and, for reasons that will be clear shortly, the \textbf{regeneration time}.  Our goal is to study the cyclic walk on the event that the cycle does not close.  Rather than study the conditional distribution, if the cycle closes we generate a new independent interchange process and run the process until the cycle closes in that one, repeating as many times as necessary.  Formally, we define the \textbf{regenerated cyclic walk} $W(s)$ as follows.  Let $\{W^{(i)}(s)\}_{i\geq 1}$ be IID cyclic walks with regeneration times $\{\tau^{(i)}_{\text{reg}}\}_{i\geq 1}$.  Then
\begin{equation}\label{eq:regen.walk.defn}
W(s):= W^{(i)}\bigg( s - \sum_{j=1}^{i-1} \tau^{(j)}_{\text{reg}} \bigg) \quad \hbox{for } s\in \bigg[ \sum_{j=1}^{i-1} \tau^{(j)}_{\text{reg}},\sum_{j=1}^{i} \tau^{(j)}_{\text{reg}} \bigg] .    
\end{equation}
That is it sequentially follows the cyclic walks $W^{(i)}(s)$ up until they close the cycle and then moves onto the next one.  Note that if $i$ is the smallest integer with $\tau^{(i)}_{\text{reg}}=\infty$ then after time $\sum_{j=1}^{i-1} \tau^{(j)}_{\text{reg}}$ then the walk evolves according to $W^{(i)}(s)$ for all future time.  We let
\[
\fR_t:=\max\bigg\{ i:\sum_{j=1}^{i} \tau^{(j)}_{\text{reg}}\leq t \bigg\}
\]
count the number of regeneration by time $t$ and let $\tau_{\text{reg}}=\tau^{(1)}_{\text{reg}}$ denote the first regeneration time. We let
\[
\alpha(t):=\inf\{s:\fR_t=\fR_s\}
\]
which is the most recent regeneration time.  We can couple the regenerated random walk in the same way where the coupling is started anew at each regeneration time.

Let $\tau _{\text{hit}}(t)$ be the first time the walk interacts with its history, that is 
\begin{equation}\label{eq:hit}
    \tau _{\text{hit}}(t) :=\inf \big\{ s>t : \exists s'\in s+\beta \mathbb Z, \ \alpha (t)<s'<t \text{ such that } \ \| W(s)-W(s') \| \le 1  \big\}.
\end{equation}
We can couple the future of the walk after time $t$ with an independent cyclic walk until $\tau _{\text{hit}}(t)$.
\begin{lem}\label{l:t.hit.couple}
Let $W(t)$ be the cyclic walk driven by $X(t)$.  Let $W'(s)$ be the cyclic walk driven by $X'(s)=X(t+s)-X(t)$.  Then for $s\in[0,\tau_{\text{hit}}(t)-t]$ we have $W(t+s)-W(t)=W'(s)$.
\end{lem}
\begin{proof}
Since the driving random walks are the same for $W'(s)$ and $W(t+s)-W(t)$, they are identical until $W(t+s)-W(t)$ is affected by its past in $[-t,0]$ which happens at $\tau_{\text{hit}}(t)-t$.
\end{proof}

In Section~\ref{sec:concatenation} we will use Lemma~\ref{l:t.hit.couple} to approximate a cyclic walk by a sum of shorter independent ones and thus show it is approximately Brownian.  Let us note that by symmetry we have that there exists $\theta_t$ such that 
\begin{equation}\label{eq:mean.covar}
\E[W(t)]=0,\qquad \hbox{Cov}(W(t)) = \theta_t I.
\end{equation}

\subsection{Constant policy}
Throughout the paper we regard the dimension $d\ge 5$, as fixed. Constants such as $C,c$ may depend on $d$ but are independent of all other parameters. These constants are regarded as generic constants in
the sense that their value may change from one appearance to the next, with the value of $C$
increasing and the value of $c$ decreasing. However, constants labeled with a fixed number,
such as $C_0$, $c_0$, have a fixed value throughout the paper.

\subsection{Deviation bounds}
In this section we give some simple a priori bounds on the walk moving too far in a short period of time. Throughout the paper $\|\cdot \|$ will be the $L^{\infty }$ norm.

\begin{lem}\label{cor:far}
For all $L,t\geq 1$,
\begin{equation}
    \mathbb P \Big( \max _{s\le t }\|W(s)\| \ge 10dLt \Big) \le Ce^{-Lt}.
\end{equation}
\end{lem}
\begin{proof}
Let $s_i=\inf\{s:\langle W(s),e_1 \rangle = i\}$ where $e_1=(1,0,\ldots,0)\in \R^d$.  Each new step further along in direction $e_1$ corresponds to a new Poisson clock ring so
\[
\P \big( s_{10LT} \leq t \big) \leq \P \big( {\rm Poisson}(t) \geq 10LT \big)  \leq e^{-Lt}.
\]
Taking a union over the co-ordinate directions and the positive and negative directions,
\[
    \mathbb P \Big( \max _{s\le t }\|W(s)\| \ge 10dLt \Big)\leq     \mathbb P \Big( \max _{s\le t }\|W(s)\|_\infty \ge 10Lt \Big)\leq e^{-Lt}. \qedhere
\]
\end{proof}

It will also be useful ensure the walk does not move to fast within a short time period.
Define the stopping time 
\begin{equation}\label{eq:tau fast}
    \tau _{\text{fast}}(L):=\inf \big\{ t>0 : \max _{s\in [t-1,t]}\|W(t)-W(s)\| \ge L \big\}.
\end{equation}

\begin{claim}\label{claim:fast}
There exists $C,c>0$, such that for $L\ge 1$  we have that 
\begin{equation}
    \mathbb P \big( \tau _{{\rm fast}} (L) \le t \big) \le Ct^2 e^{-cL}.
\end{equation}
\end{claim}

\begin{proof}
When $t\leq \beta$  we can couple the walk with a simple random walk and the result follows by standard random walk estimates.  So assume that $t\geq \beta$. First consider the problem for a non-regenerated cyclic walk $\hat{W}(s)$.  Let
\[
\cA_i=\big\{ \max_{s\in [i-1,i]}\|\hat{W}(s)-\hat{W}(i-1)\| \geq L/4 \big\}
\]
and note that by Lemma~\ref{cor:far} we have that $ \P(\cA_1)\leq Ce^{-cL}$. Next let
\[
\hat{\cC}_i= \Big\{\exists s,s'\in [(i-1)\beta ,i\beta ] \text{ with } |s-s'|\le 1 \text{ such that }\|\hat{W}(s)-\hat{W}(s')\| \geq L/2 \Big\}.
\]
By the triangle inequality, $\hat{\cC}_1 \subset \bigcup_{i=1}^{\lceil \beta \rceil} \cA_i$
and by stationary of the process have that $\P(\cA_1)=\P(\cA_i)$ and so by a union bound $\P( \hat{\cC}_1 ) \leq C \beta e^{-cL}$.

Again by stationarity we have that $\P(\cC_1)=\P(\cC_i)$.  Now, moving back to the regenerated random walk let
\[
\cC_i= \Big\{\exists s,s'\in [(i-1)\beta ,i\beta ] \text{ with } |s-s'|\le 1 \text{ such that } \|W(s)-W(s')\| \geq L/2 \Big\}.
\]
Since regeneration times happen at integer multiples of $\beta$, the regenerated and non-regenerated walks have the same distribution up to time $\beta$ and so $\P(\cC_1)=\P(\hat{\cC}_1)\leq C \beta e^{-cL}$.  If we view the regenerated cyclic walk as constructed out of a sequence of independent cyclic walks as in equation~\eqref{eq:regen.walk.defn}, we have that $\cC_i$ holds provided $\hat{\cC}_{i-\alpha((i-1)\beta)/\beta}$ holds for the $\fR_{(i-1)\beta}$th walk.  There are at most $t/\beta$ regenerations by time $t$ so by a union bound
\[
\P \bigg( \bigcup_{i=1}^{\lceil t/\beta\rceil} \cC_i \bigg) \leq \lceil t/\beta\rceil^2 \P (\hat{\cC}_1 ) \leq C t^2 \beta e^{-cL}.
\]
The lemma is completed by noting that $\{ \tau _{{\rm fast}} (L)\le t\}\subset \bigcup_{i=1}^{\lceil t/\beta\rceil} \cC_i$.
\end{proof}

\section{Induction hypothesis}\label{sec:induction}
In this section we outline our inductive hypothesis  which will be over a series of time scales $t_n=4^n t_0$ where $t_0\in (1,4]$.  Throughout the remainder of the paper we will assume that $\beta$ is a large constant and each lemma should be read as implicitly stating that ``for a sufficiently large fixed $\beta$ the following holds''.  We also fix the following constants throughout the paper
\begin{equation}
    \epsilon :=1/(300d) \quad \delta :=1/(3000d^2).
\end{equation}

Recall that $\|\cdot \|$ is the $L^{\infty }$ norm. We let $B(u,r):=\{v\in \mathbb Z ^d : \|u-v\|\le r\}$ be the ball (or more precisely the box) of radius $r$ around $x\in \mathbb Z ^d$. Next, for a subset $A \subseteq \mathbb Z ^d$ and $u\in \mathbb Z ^d$ we let $d(u,A):=\min \{ \|u-v\| : v\in A\}$. Finally, we let $N(A,r):=\{u\in\Z^d:d(u,A)\leq r\}$ be the $r$-neighbourhood of $A$.

We let $W$ be an infinite, regenerated cyclic walk. A block $B=B(x,r)$ is called \textbf{heavy} at time $t$ if
\begin{equation}
    \big| \big\{ u\in B : \ \exists \, \alpha (t) \le s \le t, \ W(s)=u \big\} \big| \ge r^{5/2}.
\end{equation}
Note that typically a random walk stays $O(r^2)$ time inside $B(x,r)$ so if the cyclic walk behaves like random walk then we do not expect to encounter heavy blocks above a logarithmic scale. 

We say that a time $t$ is \textbf{relaxed} with respect to the walk $W$ if the following holds:
\begin{enumerate}
    \item 
    For all $r\ge \beta ^\delta $ the block $B(W(t),r)$ is not heavy.
    \item
    For all $\alpha (t) \le s \le t-\beta ^{3\delta }$ we have that $\|W(t) -W(s)\|\ge \beta ^\delta $.
\end{enumerate}

We let $\mathcal T$ be the set of relaxed times.

\begin{definition}
A path of length $t\in(4^n,4^{n+1}]$ is relaxed if
\begin{equation}
    \big| \mathcal T \cap [0,t] \big| \ge \big(0.9+1/n\big) t.
\end{equation}
\end{definition}
The required fraction of relaxed times decreases slightly as $n$ increases to allow some room for errors in the induction step.

We define the inductive failure probability bound at time $t_n$ to be
\begin{equation}
    p_n=p_n(\beta ,t_0 ):= \begin{cases}
    e^{-n^3}  &  \text{when } t_n < \beta \\
    e^{-n^2} &\text{when }  t_n \ge  \beta     \end{cases}
\end{equation}

The following 5 assumptions are the inductive hypothesis at time $t_n$.

{\bf (1) Relaxed paths:}
    The path up to time $t_n$ is relaxed with probability at least $1-p_n$.

{\bf (2) Transition probabilities:}
    For all $1 \le t\le t_n$ and $u\in \mathbb Z^d $ we have that $$\mathbb P (W(t)=u) \le t^{-d/2+0.1}.$$

{\bf (3) Traveling far:}
We have that 
\begin{equation}
    \mathbb P \Big( \sup _{s\le t_n} \|W(s)\| \ge n^2\sqrt{t_n} \Big) \le p_n.
\end{equation}
    
{\bf (4) Approximately Brownian:}
    There exists a constant $\sigma=\sigma (t_n)$ with
    \begin{equation}
    |\sigma -\sqrt{2} |\le \max\{3\beta ^{-1/16} -t_n^{-1/16},0\}
\end{equation}
 such that the following holds. There is a coupling of $W$ and a Brownian motion $B$ such that  
    \begin{equation}
        \mathbb P \Big( \max _{t\le t_n} | W(t)-\sigma B(t)| \ge t_n^{2/5} \Big) \le p_n.
    \end{equation}
    
{\bf (5) Pair proximity property:}
Let $W_1$ and $W_2$ be two cyclic walks in the same interchange environment, starting from $u_1,u_2\in \mathbb Z ^d$ at some cyclic times $q_1,q_2\le \beta $. Recall that these are the cyclic walks that satisfy $W_i(0)=u_i$ and that jump along a neighbouring edge at time $s$ if this edge rings at time $s+q_i (\!\!\!\! \mod \beta )$. Suppose that the walks run for a slightly reduced time $t_n':=t_n/n^4$. As before, the walk $W_i$ closes if there exists $ s \le t_n'$ with $s\in \beta \mathbb Z$ such that $W_i(s)=W_i(0)$. We say that $W_1$ and $W_2$ merge if there are $s_1,s_2\le t_n'$ such that $(s_1+q_1)-(s_2+q_2)\in \beta \mathbb Z$ and $W_1(s_1)=W_2(s_2)$. This happens when the first walk reaches $u_2$ at time $q_2-q_1+k\beta$ for $k\in \mathbb Z$ or the second walk reaches $u_1$ at time $q_1-q_2+k\beta $. We let $\Omega =\Omega (u_1,u_2,q_1,q_2)$ be the event that $W_1$ and $W_2$ did not merge and non of them closed up to time $t_n'$. Finally, we write $\fP_n=\fP_n(u_1,u_2,q_1,q_2)$ for the set of times $W_2$ is proximate to the trajectory of $W_1$,
\[
\fP_n=\big\{ s_2 \le t_n'  : \exists \, s_1\le t_n', \ \|W_1(s_1)-W_2(s_2)\| \le 1.9^n \big\}.
\]
The induction assumption is that for any choice of $u_1,u_2,q_1,q_2$ we have that
    \begin{equation}\label{eq:path pair}
       \mathbb P \big(\Omega  \text{ and } | \fP_n | \ge 3.9^n \big) \le p_n .
    \end{equation}
In words, as long as the paths don't close or merge, their trajectories will typically be far with high probability.

\begin{thm}\label{t:inductive}
There exists $\beta_0$ and $n_0$ such that for all $\beta\geq \beta_0$ and $n\geq n_0$ the 5 inductive hypothesises hold.
\end{thm}

Using this theorem we can easily prove Theorem~\ref{t:main.thm}.

\begin{proof}[Proof of Theorem~\ref{t:main.thm}]
A cycle can only be closed at an integer multiple of $\beta$ and therefore by inductive hypothesis
\[
\P[\tau_{\text{reg}}<\infty] \leq \sum_{i=1}^\infty \P[W(i\beta)=0] \leq \sum_{i=1}^\infty (i\beta)^{-d/2+0.1}< 2/\beta <1
\]
for large $\beta$. Theorem~\ref{t:main.thm} follows immediately from the last estimate and the zero-one law.
\end{proof}

\section {Avoiding the history from Relaxed Times}\label{sec:relaxed}

In the next four sections we will assume the induction hypothesis holds for all $1<t_0\le 4$ and all $n'<n$. In particular, it holds at time $4^n$. The following lemma shows that at relaxed times we have a good probability of not interacting with our past.

\begin{lem}\label{lem:relaxed4}
Suppose that $t>0$ is a relaxed time. Then,
\begin{equation}
    \mathbb P \big( \tau _{ \rm hit }(t)\ge t+4^n \ | \ \mathcal F _t \big) \ge 1-C\beta^{-c} .
\end{equation}
\end{lem}

For the proof of the lemma we will need the following claims.

\begin{claim}\label{claim:A}
Let $L\ge 1$ and let $A\subseteq \big\{ u\in \mathbb Z ^d: \|u\|\ge L \big\}$. Suppose that for all $r\ge L$ we have that $\big| A\cap [-r ,r ]^d \big| \le r^{8/3}$. Then
\begin{equation}
    \mathbb P \big( \forall s \le 4^n, \ W(s)\notin A  \big) \ge 1-CL^{-c}.
\end{equation}
\end{claim}

\begin{proof}
Let $\ell _0 := \lfloor \log _{4} L \rfloor $ and for all $ \ell _0 \le \ell \le n$ consider the set $A_\ell :=N(A,\ell ^2)$. By the traveling far property in the induction hypothesis we have that $\mathbb P (\|W(s)\| \ge \ell ^22^{\ell} )\le e^{-c\ell ^2}$ for all $4^{\ell -1} \le s\le 4^{\ell}$. Thus, using the inductive transition probabilities and the assumption on the density of $A$ we obtain for all $4^{\ell -1}\le s \le 4^{\ell}$,
\begin{equation}
\begin{split}
    \mathbb P (W(s)\in A_\ell ) &\le \mathbb P (\|W(s)\| \ge \ell ^2 2^\ell) +\mathbb P \big( W(s) \in A_\ell  \cap [-\ell ^2 2^{\ell} ,\ell ^2 2^{\ell }]^d \big)\\
    &\le e^{-c\ell ^2 }+ C\ell ^C 2^{8\ell/3 } s^{-d/2+0.1} \le 2 ^{-2.1\ell},
\end{split}
\end{equation}
where in the last inequality we used that $d\ge 5$. Thus,
\begin{equation}
\begin{split}
    \mathbb P \big( \exists s\in \mathbb N \cap [4^{\ell -1},4^{\ell }], \ W(s)\in A_\ell   \big) \le e^{-c\ell }.
\end{split}
\end{equation}
Next, consider the event $ \mathcal D _\ell :=  \{ \tau _{\text{fast}} (\ell ^2) > 4^{\ell } \} $ and note that by Lemma~\ref{claim:fast} we have that $\mathbb P (
\mathcal D _\ell )\ge 1- e^{-c\ell ^2}$. Moreover, on $\mathcal  D_\ell $ for all $4^{\ell -1} \le s \le 4^{\ell}$ such that $W(s)\in A$ we have that  $W(\lceil s \rceil )\in A _\ell  $. It follows that
\begin{equation}
    \mathbb P \big( \exists s\in [4^{\ell -1} ,4^{\ell }], \ W(s)\in A  \big)  \le \mathbb P \big( \exists s\in \mathbb N \cap [4^{\ell -1} ,4^{\ell }], \ W(s)\in A_\ell   \big)+ \mathbb P (\mathcal D _\ell ^c) \le e^{-c\ell}.
\end{equation}
Finally, using that $A \subseteq  \{u\in \mathbb Z ^d :  \|u\|\ge L \}$ and a union bound we obtain
\begin{equation}
\begin{split}
    \mathbb P \big( \exists s\le 4^{n}, \ W(s)\in A  \big) \le \mathbb P \Big( \max _{s\le L} \|W(s)\| \ge L \Big) &+ \sum _{\ell =\ell _0  } ^{n} \mathbb P \big( \exists s\in [4^{\ell -1},4^{\ell }], \ W(s)\in A  \big)\\
    &\le CL^{-1} +C\sum _{\ell =\ell _0 } ^{\infty } e^{-c\ell } \le CL ^{-c},
\end{split}
\end{equation}
where in the second inequality we used the traveling far property of the inductive assumption.
\end{proof}

\begin{claim}\label{claim:W big}
For any $t\le 4^n$ and $t^{2/5}\le L\le t^{1/2}$ we have that 
\begin{equation}\label{eq:W large}
    \mathbb P \big(\forall t\le s \le 4^n, \ \|W(s)\|\ge L \big) \ge 1-CL^{d-2}t^{-(d-2)/2}.
\end{equation}
\end{claim}

\begin{proof}
We prove the claim using the inductive Brownian approximation. Let $\ell _0:=\lfloor \log _4 t \rfloor $ and for $\ell \ge \ell _0$ let $B_\ell $ be the Brownian motion from the induction hypothesis at time $4^{\ell}$. Define the event
\begin{equation}
    \cD := \bigcap _{\ell = \ell _0 } ^n \bigg\{ \max _{s\le 4^\ell }\|W(s)-\sigma_\ell B_\ell (s)\| \le 4^{2\ell /5} \bigg\}
\end{equation}
and note that by the induction hypothesis we have that $\mathbb P (\cD) \ge 1-Ce^{-c\ell _0^2}$. Next, define
\begin{equation}
    \mathcal C := \bigcap _{\ell =\ell _0  } ^n \big\{ \inf_{s\in [4^{\ell -1},\le 4^{\ell}]} \ \|\sigma_\ell B_\ell (s)\| \ge 2\max (4^{2\ell /5},L) \big\}.
\end{equation}
Let $\ell _1$ be the first integer for which $4^{2\ell /5}>L$. By Lemma~\ref{l:BM.ball.exit} we have that
\begin{equation}
\begin{split}
    \mathbb P (\mathcal C ^c) \le \sum _{\ell =\ell _0}^{\ell _1} \mathbb P \Big( \inf_{ s>4^{\ell -1}} \ \|\sigma_\ell B_\ell (s)\| \ge 2L  \Big)+\sum _{\ell =\ell _1}^{n} \mathbb P \Big( \inf_{ s>4^{\ell -1}} \ \|\sigma_\ell B_\ell (s)\| \ge 2\cdot 4^{2\ell /5}  \Big)\\
    \le C\sum _{\ell =\ell _0}^{\ell_1 } L^{d-2}2^{-(d-2)\ell }  + C\sum _{\ell =\ell _1}^{n} 4^{-(d-2)\ell /10}  \le CL^{d-2}t^{-(d-2)/2}
\end{split}
\end{equation}
which completes the proof since on $\cC \cap \cD $ the event in \eqref{eq:W large} holds. 
\end{proof}

We now prove Lemma~\ref{lem:relaxed4}.

\begin{proof}[Proof of Lemma~\ref{lem:relaxed4}]
Without loss of generality, suppose that $t<\tau _{\text{reg}}$. Indeed, the same holds between any two regeneration times since we ignore the past before the last regeneration in the definition of $\tau _{\text{hit}}(t)$. Let $\tilde{W}$ be a regenerated walk independent of $W$ and let $W'(s):=W(t)+\tilde{W}(s)$ for all $s>0$.  By Lemma~\ref{l:t.hit.couple} we can couple $W$ and $W'$ such that for all $t\le s \le \tau _{\text{hit}} (t)$ we have that $W(s)=W'(s-t)$. Next, define the set 
\begin{equation}
    A:=\big\{ u\in \mathbb Z ^d : \|u-W(t)\|\ge \beta^{\delta }/2 \text{ and } \exists s\in [0,t], \  \|W(s)-u\| \le 1 \big\}
\end{equation}
and note that for all $r\ge \beta ^{\delta }/2$ we have 
\begin{equation}
 \big| A \cap B(W(t),r) \big| \le 2d \cdot  \big| \big\{ u\in B(W(t),2r) : \exists s\le t,  \ W(s)=u  \big\} \big| \le  Cr^{5/2},   
\end{equation}
where the last inequality holds as $t$ is relaxed. Thus, using Claim~\ref{claim:A} and that $\tilde{W}$ is independent of $\mathcal F _t$ we obtain 
\begin{equation}\label{eq:12}
    \mathbb P \big( \forall  s\le 4^n, \ W'(s)\notin A  \ | \ \mathcal F _t \big) \ge 1-\beta^{-c}.
\end{equation}
Moreover, by Claim~\ref{claim:W big}
\begin{equation}\label{eq:13}
    \mathbb P \big( \forall \beta^{3\delta } \le s \le 4^n, \ \|W'(s)-W(t)\|\ge \beta^{\delta } \ | \ \mathcal F_t   \big) \ge 1-\beta^{-c}.
\end{equation}
It suffices to show that on the intersection of \eqref{eq:12} and \eqref{eq:13} we have that $\tau _{\text{hit}}(t)\ge t+4^n$. To this end, by the definition of $\tau _{\text{hit}}(t)$ we have that
\begin{equation}
 \|W'(\tau _{\text{hit}}(t)-t)-W(s)\| \le 1   \text{ for some }  s\le t \text{ with } s\in \tau _{\text{hit}}(t)+\beta \mathbb Z.
\end{equation}
If $\|W'(\tau _{\text{hit} }(t)-t)-W(t)\| \ge \beta^{\delta }/2$ then $W'(\tau _{\text{hit} }(t)-t)\in A$. Thus, on the event in \eqref{eq:12} we have $\tau _{\text{hit}}(t)-t\ge 4^n$. If $\|W'(\tau _{\text{hit} }(t)-t)-W(t)\| < \beta^{\delta }/2$ then on the event in \eqref{eq:13} we have $\tau _{\text{hit}}(t)-t < \beta^{3\delta }< \beta /2$ and therefore $s< t-\beta/2$ (since $s\in \tau _{\text{hit}}(t)+\beta \mathbb Z$). However, by the definition of relaxed time, for such $s$ we have $\|W(s)-W(t)\|\ge \beta^{\delta }$ which is a contradiction since $\|W'(\tau _{\text{hit} }(t)-t) -W(s)\|\le 1$ and $\|W'(\tau _{\text{hit} }(t)-t)-W(t)\| < \beta^{\delta }/2$. It follows that $\tau _{\rm hit}(t)\ge t+4^n$ on the intersection of the events in 
\eqref{eq:12} and \eqref{eq:13}.
\end{proof}

\section{The escape algorithm}\label{sec:escape}

We will use the results of this section in order to escape fast from (possibly) messy situations. Assume that the induction hypothesis holds at all levels $n'< n$. Throughout the paper we let $\mathcal B _m$ be the set of blocks of the form $u+[0,2^m)^d$ where $u\in 2^m\mathbb Z ^d$. The set $\mathcal B _m$ is a partition $\mathbb Z ^d$ into disjoint blocks and $\mathcal B_m$ is a refinement of $\mathcal B _{m+1}$. We fix $j=j_n$ to be the unique integer for which $n^{20}\le 2^j<2n^{20}$. Finally, let $\gamma =\gamma (d)>2$ be a sufficiently large constant that will be determined later. For a time $t>0$ define the escape event
\begin{equation}\label{eq:def of E}
    \mathcal E (t):= \big\{ \forall 4^j \le s \le  4^n , \ \|W(t+s)-W(t)\| \ge \sqrt{s}/n \big\}.
\end{equation}
Finally, define the escape stopping time
\begin{equation}
    \tau _{\text{esc}}:=\inf \Big\{ t>0 \ : \  t \text{ is the first time we enter a block }B\in \mathcal B _j  \text{ and } \mathbb P (\mathcal E (t) \mid \mathcal F _t ) \le n^{-\gamma } \Big\}.
\end{equation}
In words, before $\tau _{\text{esc}}$ every time we enter a block $B\in \mathcal B _j$ for the first time, we have a chance of at least $n^{-\gamma }$ to escape far in a short amount of time.

The main result of this section is the following theorem.
\begin{thm}\label{thm:escape}
We have that $\mathbb P (\tau _{\rm esc}<t_n\wedge \tau _{\rm reg}) \le e^{-cn^3}$.
\end{thm}

Theorem~\ref{thm:escape} follows immediately from the following proposition. From now on we fix a block $B=u+[0,2^j)^d\in \mathcal B _j$ and let $\xi _B$ to be the first time we enter $B$. 

\begin{prop}\label{prop:fix B}
We have that 
\begin{equation}
    \mathbb P \Big( \xi _B \le t_n \wedge \tau _{\rm reg} \ \ {\rm and } \ \ \mathbb P \big( \mathcal E (\xi _B ) \ | \ \mathcal F _{\xi _B} \big) \le n^{-\gamma } \Big) \le e^{-cn^3}.
\end{equation}
\end{prop}

\begin{figure}[htp]
    \centering
    \includegraphics[width=17cm]{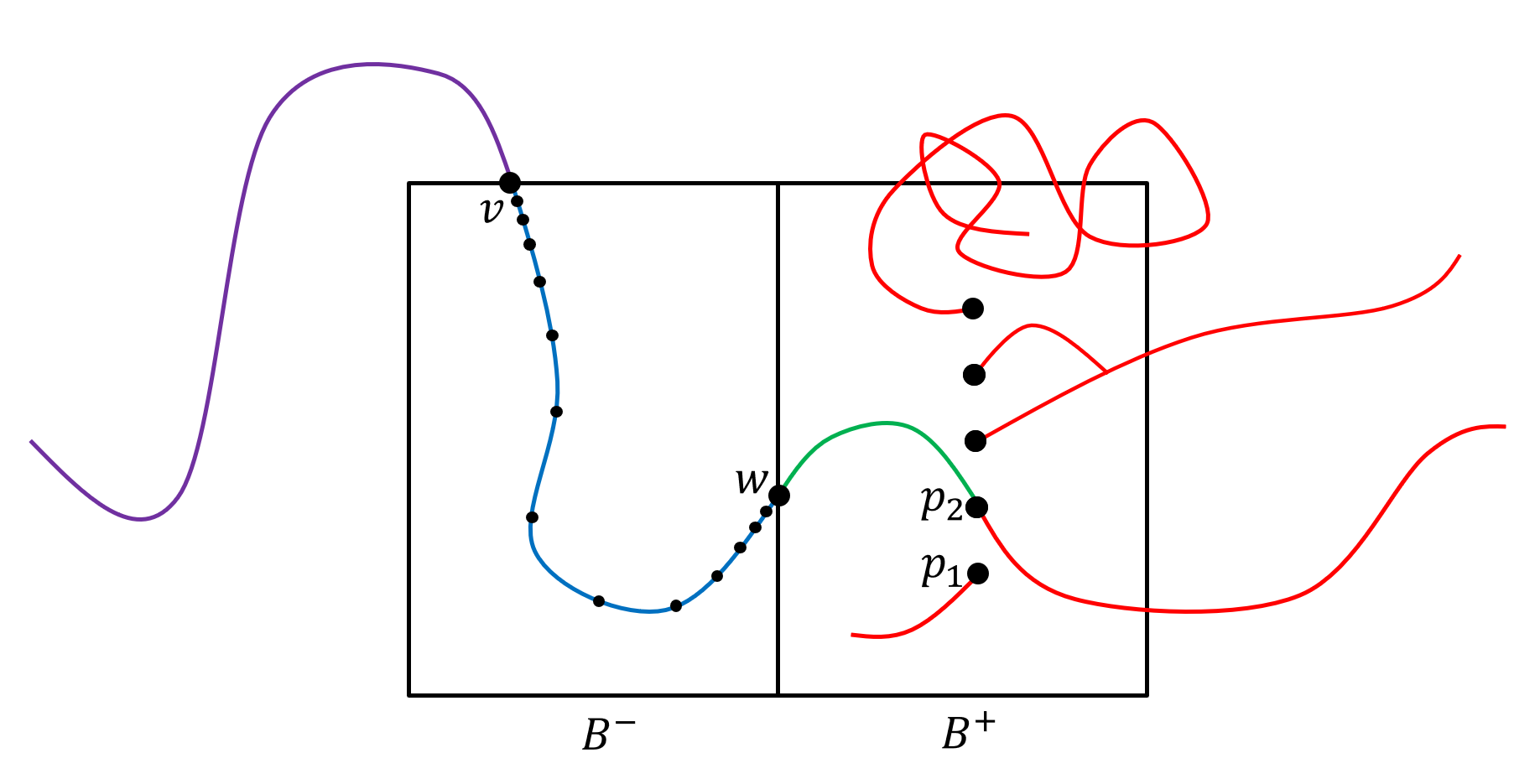}
    \caption{The escape algorithm. Corollary~\ref{cor:A} shows that with very high probability at least one of the red paths will escape far. Most likely, this is true after conditioning on the history before entering $B$ (the purple path). Once we enter $B$, we reveal the red paths and find the successful starting point ($p_2$). Then, we use Lemma~\ref{lem:connect} to connect $v$ to this starting point at the right time and escape. For the proof of Lemma~\ref{lem:connect}, we first use Claim~\ref{claim:green} that shows that with positive probability, the green path will avoid the red paths and connect to $B^-$ in a short amount of time. Then, we use Claim~\ref{claim:chain} to connect $v$ to $w$ while staying in $B^-$. This is done by splitting the blue path into $O(\log n)$ many paths of different length scales and use the Brownian approximation in each scale.
    }
    \label{fig:escape}
\end{figure}

Theorem~\ref{thm:escape} follows from Proposition~\ref{prop:fix B} and a union bound over all blocks $B\in \mathcal B_j$ at distance at most $5^n$ from the origin.

We next define the event $\mathcal A =\mathcal A(B,\beta ')$ for an integer $\beta '\le \beta $. Intuitively, $\mathcal A $ says that the block $B$ is escapable. That is, if the walk enters $B$ for the first time then it has a chance of at least $n^{-\gamma}$ to escape fast.

We let $B^-:=u+[0,2^{j-1})\times[0,2^{j})^{d-1}$ and $B^+:=u+[2^{j-1},2^j)\times[0,2^{j})^{d-1}$ be the left and right halves of the box $B$ respectively. Let $p_1,\dots ,p_{n^3}$ be equally spaced points in the interior of $B^+$. To make things concrete, let 
\begin{equation}
    p_i:=u+ 2^{j-2}(3,2,\dots ,2) + i\lfloor n^{17}/3 \rfloor  e_2  \quad i \le n^3.
\end{equation}
See Figure~\ref{fig:escape}. For each one of these points we consider the cyclic random walk $W_i(t)$ starting from $p_i$ at the cyclic time $\beta '$.  We define inductively a sequence of stopping times $\tau _i$ and expose the walk $W_i(t)$ up to time $\tau _i$. Suppose that $\tau _1,\dots , \tau _{i-1}$ were defined. We define $\tau _i$ to be the first time $t>0$ for which one of the following occurs:
\begin{enumerate}
    \item 
    The walk reaches $0$.
    \item
    We have $t\ge \log \beta $ and $\|W_i(t)-p_i\|\le 2\sqrt{t}/\log t$.
    \item
    We have $t\ge \log \beta$ and $\tau ^{(i)}_{\rm fast}(\log ^2 t)\le t$, where $\tau _{\rm fast}^{(i)}$ is defined as in \eqref{eq:tau fast} with $W_i$.
    \item 
    We have $\| W_i(t)-p_\ell \| \le n^{17}/6$ for $\ell \neq i$.
    \item
    For some $\ell <i$ and $ s\le \tau _\ell $ we have $\| W_i(t)-W_\ell(s)\| \le 1$.
    \item 
    We have that $d(W(t),B^-)\le 2^{j-3}$.
    \item 
    time reaches $4^n$.
\end{enumerate}
We define $ \mathcal A _i=\mathcal A_i (\beta '):=\{\tau _i=4^n\}$ and let $\mathcal A :=\bigcup _{i\le n^3} \mathcal A _{i}$.

\begin{lem}\label{lem:Ai}
For all $i\le n^3$ we have that 
\begin{equation}
    \mathbb P \big( \mathcal A_{i}  \ \big| \ W_\ell (s), \   s\le \tau _\ell , \ \ell <i \big) \ge c. 
\end{equation}
\end{lem}

\begin{proof}
Let $W_i'$ be a regenerated cyclic walk starting from $p_i$ and independent of $(W_\ell (s), \ s\le \tau _\ell , \ell <i)$. We can couple $W_i$ and $W_i'$ such that $W_i(s)=W'_i(s)$ for all $s\le \tau _i$. Define the event
\begin{equation}
    \mathcal B := \bigcap _{m=m_1  }^n \big\{ \forall 4^m\le s \le 4^n, \ \|W_i'(s)-p_i\|\ge 2^{m+3}/m \big\},
\end{equation}
where $m_1:=\lfloor (\log \log \beta ) /2\rfloor $. Clearly, on $\mathcal B$ the stopping time $\tau _i$ will not stop because of condition (1). Moreover, on this event the walk $W_i'$ will not regenerate. By Claim~\ref{claim:W big} we have that 
\begin{equation}
    \mathbb P (\mathcal B ^c) \le C\sum _{m=m_1}^{\infty } m^{-3} \le C /(\log \log \beta ), 
\end{equation}
where in here we also used that $d\ge 5$. 

Next, define the set 
\begin{equation}
    A:=\{0\}\cup \big\{u\in \mathbb Z ^d :  \exists \ell <i,\  s\le \tau _\ell ,\ \|u-W_\ell (s)\| \le 1   \big\}.
\end{equation}
and the event $\mathcal C := \{\forall s\le 4^n : W_i'(s)\notin A\}$. On $\mathcal C$ we will not stop because of conditions (1) and (5). Note that by condition (4), any $u\in A$ satisfies $\|u-p_i\| \ge n^{17}/7$. Moreover, by conditions (2) and (3), for each $\ell <i$, the number of vertices visited by the walk $W_\ell$ in the block $p_\ell +[-r,r]^d$ is at most $r^2 \log ^Cr$ for all $r\ge \log \beta $. Thus, for all $n^{17}/7 \le r\le 2^j$ we have that
\begin{equation}
    \big| A\cap \big( p_i+[-r,r]^d \big)  \big| \le n^3 j^C4^j \le r^{8/3}
\end{equation}
and for all $r\ge 2^j$ we have 
\begin{equation}
    \big| A\cap \big( p_i+[-r,r]^d \big)  \big| \le n^3 r^2 \log ^C r \le r^{8/3}.
\end{equation}
We can now use Claim~\ref{claim:A} to obtain that $\mathbb P \big( \mathcal C  \ \big| \ W_\ell (s), \   s\le \tau _\ell , \ \ell <i \big)\ge 1-n^{-c}$.

Next, define the event 
\begin{equation}
    \mathcal D := \bigcap _{m=m_1}^n \big\{ \tau _{\text{fast}}^{(i)}(m^2) \ge 4^m  \big\}.
\end{equation}
By Claim~\ref{claim:fast} we have that $\mathbb P (\mathcal D) \ge 1-1/\log \beta $. On $\mathcal D$ we will not stop because of condition (3).

Let $B(s)$ be the Brownian motion from the inductive assumption at time $j4^j$ for $W_i'$. By the inductive assumption we have that $\mathbb P (\mathcal H)\ge 1-Ce^{-cj^2}$ where we let $\mathcal H$ be the event that the coupling of $W_i'$ and $B_j$ was successful. Finally, consider the event 
\begin{equation}
    \mathcal I :=\Big\{ \forall s\le j4^j,\   d(B_j(s),B^-) \ge 0.2\cdot 2^j  \text{ and } \forall \ell \neq i, \ \|B_j(s)-p_\ell \| \ge n^{17}/5
     \Big\}.
\end{equation}
We clearly have that $\mathbb P (\mathcal H)\ge c$. Moreover, on $\mathcal H \cap \mathcal I \cap \mathcal B$ we will not stop because of conditions (4) and (6).

We obtain that the event $\mathcal J =\mathcal B \cap \mathcal C \cap \mathcal D \cap \mathcal I \cap \mathcal H $ satisfies $\mathbb P \big( \mathcal J  \ \big| \ W_\ell (s), \   s\le \tau _\ell , \ \ell <i \big)\ge  c$ and that on $\mathcal J$ we have $\tau _i=4^n$. This finishes the proof of the lemma.
\end{proof}

The following corollary is immediate. 
 
\begin{cor}\label{cor:A}
We have that $\mathbb P \big( \mathcal A (\beta ') \big) \ge 1-e^{-cn^3}$.
\end{cor}

Next, let $\mathcal G  (\beta ')$ be the sigma algebra generated by $(W _i(s), \ s\le \tau _i, \ i\le n^3)$. We also let $t'$ be the last cyclic integer before $\xi _B +4^j$ and  define the event
\begin{equation}
    \mathcal K (\beta ') :=\big\{ \xi _B \le t_n \wedge \tau _{\rm  reg}, \ W(\xi _B) \in B^- \text{ and } \ t' \equiv \beta ' (\!\!\!\!\!\!\mod \beta ) \big\}
\end{equation}
for any integer $\beta ' \le \beta $. Corollary~\ref{cor:A} shows that with high probability there is $i\le n^3$ such that the cyclic walk starting from $p_i$ escapes fast. The next lemma shows that with non negligible probability we can connect to $p_i$ at the right time and escape. 

\begin{lem}\label{lem:connect}
For all integers $\beta '\le \beta $ and  $i\le n^3$, on the event $\mathcal K (\beta ')$  we have
\begin{equation}
       \mathbb P \big( W(t')=p_i \ | \ \mathcal F _{\xi _B}, \ \mathcal G (\beta ') \big) \ge 2n^{-\gamma }. 
\end{equation}
\end{lem}

Using Lemma~\ref{lem:connect} and Corollary~\ref{cor:A} we can easily prove Proposition~\ref{prop:fix B}.

\begin{proof}[Proof of Proposition~\ref{prop:fix B}]
Note that for any integer $\beta '\le \beta $ and all $i\le n^3$ we have that 
\begin{equation}
    \mathcal K(\beta ') \cap  \big\{ W(t')=p_i \big\} \cap \mathcal A_{i}(\beta ') \subseteq \mathcal E (\xi _B).
\end{equation}
Thus, by Lemma~\ref{lem:connect} we have on the event $\mathcal K(\beta ')$ 
\begin{equation}\label{eq:5630}
\begin{split}
    \mathbb P \big( \mathcal E (\xi _B)  \ | \ \mathcal F_{\xi _B}  \big) &\ge \sum _{i=1}^{n^3} \mathbb P \big( \mathcal A_{i}(\beta '), \ W(t')=p_i  \ | \ \mathcal F _{\xi _B}\big)\\
    &= \sum _{i=1}^{n^3} \mathbb E \Big[ \mathds 1 _{\mathcal A _{i}(\beta ')} \cdot \mathbb P \big(  W(t')=p_i  \ | \ \mathcal F _{\xi _B}, \ \mathcal G (\beta ')\big) \  \big| \ \mathcal F _{\xi _B} \Big]  \\
    &\ge  2n^{-\gamma } \sum _{i=1}^{n^3} \mathbb P \big( \mathcal A _{i}(\beta' ) \ | \ \mathcal F _{\xi _B} \big)\ge 2n^{-\gamma } \cdot \mathbb P \big( \mathcal A (\beta ')  \ | \ \mathcal F _{\xi _B} \big).
\end{split}
\end{equation}

Thus, letting $\mathcal Q := \big\{ \mathbb P \big( \mathcal A (\beta ')  \ |  \ \mathcal F _{\xi _B\wedge t_n \wedge \tau _{\rm reg} } \big) >  1/2  \big\}$ we obtain that on the event $\mathcal K (\beta ') \cap \mathcal Q$ we have $ \mathbb P \big( \mathcal E(\xi _B) \mid \mathcal F _{\xi _B} \big)  >n^{-\gamma }$. Hence, by Markov's inequality and Corollary~\ref{cor:A}  
\begin{equation}
\begin{split}
\mathbb P \Big( \mathcal K (\beta ') \text{ and } \mathbb P \big( \mathcal E(\xi _B)& \mid \mathcal F _{\xi _B} \big) \le n^{-\gamma } \Big) \le\mathbb P (\mathcal Q ^c ) \\
&=\mathbb P \big( \mathbb P \big( \mathcal A (\beta ') ^c  \ |  \ \mathcal F _{\xi _B\wedge t_n \wedge \tau _{\rm reg} } \big)\ge 1/2 \big) 
\le 2 \cdot \mathbb P \big( \mathcal A (\beta ')^c \big) \le e^{-cn^3}. 
\end{split}
\end{equation}
Summing over integers $\beta '\le \beta $ we have
\begin{equation}
    \mathbb P \Big( \xi _B \le t_n \wedge \tau _{\rm reg}, \ W(\xi _B) \in B^- \ \text{and} \ \mathbb P \big( \mathcal E (\xi _B) \mid \mathcal F _{\xi _B} \big) \le n^{-\gamma }  \Big) \le e^{-cn^3},
\end{equation}
Using exactly the same arguments when $B^-$ is replaced with $B^+$ finishes the proof of the proposition.
\end{proof}

We turn to prove Lemma~\ref{lem:connect}. To this end we need the following claims.

\begin{claim}\label{claim:chain}
Let $v,w\in \mathbb Z ^d$ such that
\begin{equation}
 v\in \partial B \cap B^-, \quad w\in u+\{2^{j-1}-1\}\times [2^{j-2}, 3\cdot 2^{j-2}]^{d-1}
\end{equation}
where $\partial B$ is the set of vertices in $B$ with a neighbour outside of $B$. Let $W_v(t)$ be the cyclic random walk starting at $v$. For all $4^{j-1} \le t \le 4^{j}$ we have that
\begin{equation}
    \mathbb P \big( W_v(t)=w \text{ and } W_v(s)\in B^-  \text{ for all }s\le t \big) \ge n^{-\gamma +1}. 
\end{equation}
\end{claim}

We turn to state the second claim. To this end, recall that for $i\le n^3$, $W_i(t)$ is the cyclic random walk that starts in $p_i$ at time $\beta '$. We can also expose $W_i$ backward in time and so $W_i(s)$ is defined for all $s\in \mathbb R$. Let 
\begin{equation}
    \zeta _i:= \sup \{t<0: W_i(t)\notin u+[2^{j-1},7\cdot 2^{j-3}]\times [2^{j-2},3\cdot 2^{j-2}]^{d-1} \}.
\end{equation}

\begin{claim}\label{claim:green}
We have that 
\begin{equation}
    \mathbb P \big( \zeta _i \ge -4^{j-1} \text{ and }  W_i(\zeta _i^-) \in B^- \mid  \mathcal G (\beta ') \big) \ge c. 
\end{equation}
\end{claim}

Using these claims we can prove Lemma~\ref{lem:connect}.

\begin{proof}[Proof of Lemma~\ref{lem:connect}]
Let $v\in \partial B \cap B^-$ and $w\in u+\{2^{j-1}-1\}\times [2^{j-2}, 3\cdot 2^{j-2}]^{d-1}$. By Claim~\ref{claim:chain}, on the event $\mathcal K (\beta ') \cap \{W(\xi _B)=v,\ \zeta _i \ge -4^{j-1}, \ W_i(\zeta _i^-)=w\}$ we have that
\begin{equation}\label{eq:use claim4.6}
    \mathbb P \big( W(t'+\zeta _i^-)=w \ | \ \mathcal F _{\xi _B}, (W_i(t),  t\in [\zeta _i,0])   \big)  \ge n^{-\gamma +1}.
\end{equation}
Indeed, let $W_v$ be a cyclic walk starting from $v$ that is independent of $\mathcal F_{\xi _B}$ and $(W_i(t),  t\in [\zeta _i,0])$. On the event 
\begin{equation}
\{ W_v(t'-\xi _B+\zeta _i)=w  \text{ and } W_v(s)\in B^- \text{ for all } s\le t'-\xi _B+\zeta _i\}    
\end{equation}
We can couple $W$ and $W_v$ such that $W(\xi _B+s)=W_v(s)$ for all $s< t'-\xi _B+\zeta _i$. We note that there is a subtle point in here that requires some attention. The sigma algebras $\mathcal F _{\xi _B}$ and $\sigma (W_i(t),  t\in [\zeta _i,0])$ contain some information about edges that enter $B^-$ and this information may ruin the coupling with $W_v$. However, the only information on such edges (except for the edges connected to $v$ and $w$) is that the edge does not ring in a certain time interval since otherwise the walk $W$ would have entered before $\xi _B$ or $W_i$ would have exited after $\zeta _i$. Thus, on the event that $W_v$ stays in $B^{-}$ the information revealed on this edges is compatible. Note also that $W_v$ does not return to $v$ at a cyclic time before $t'-\xi _B+\zeta _i$ since it is a cyclic walk and therefore, the ring at $\xi _B$ on the edge containing $v$ will not force $W$ out of $B^-$. The same holds at time $t'+\zeta _i$ on the edge containing $w$. 

Finally, using \eqref{eq:use claim4.6} we obtain that in the event $\mathcal K (\beta   ') \cap \{W(\xi _B)=v\}$
\begin{equation}
\begin{split}
       \mathbb P \big(& W(t')=p_i \mid \mathcal F _{\xi _B}, \ \mathcal G  (\beta ') \big) \\
      &  \ge \sum _w \mathbb P \big( \zeta _i \le 4^{j-1} \text{ and }  W_i(\zeta _i^-) =w , \ W(t'+\zeta _i^-)=w \mid  \mathcal F _{\xi _B},\  \mathcal G (\beta ') \big)   \\
      & \ge n^{-\gamma +1} \sum _w \mathbb P \big( \zeta _i \le 4^{j-1} \text{ and }  W_i(\zeta _i^-) =w \mid  \mathcal G (\beta ') \big) \ge cn^{-\gamma +1}\ge 2n^{-\gamma 
      },
\end{split}
\end{equation}
where the sum is over $w\in u+\{2^{j-1}-1\}\times [2^{j-2}, 3\cdot 2^{j-2}]^{d-1}$ and where the third inequality is by Claim~\ref{claim:green}.
\end{proof}

We turn to prove Claim~\ref{claim:chain}.

\begin{proof}[Proof of Claim~\ref{claim:chain}]
If $4^j\le \beta $ then the claim follows from a standard random walk estimates.

Suppose next that $4^j\ge \beta $. Let $q=u+2^{j-2}(1,2\dots ,2)$ be the center of $B^-$ and define the continuous curve $f :[0,t]\to B^-$ in the following way
\begin{equation}
    f (s):=\begin{cases}
    \sqrt{2s/t}\cdot 
q+\big( 1-\sqrt{2s/t} \, \big)v, \quad &s\le t/2 \\
\sqrt{2(t-s)/t}\cdot 
q+\big(1-\sqrt{2(t-s)/t} \, \big)w, \quad &t/2 \le s\le t
    \end{cases}.
\end{equation} 
The idea will be to use the Brownian approximation to show that $W_v(s)$ is close to $f (s)$ in any dyadic scale with positive probability.

Let $m:=\lfloor \log _4 (t / \sqrt{\beta})  \rfloor $. Note that $m$ is a positive integer since $t\ge 4^{j-1}$ and $\beta \le 4^j$. We define a sequence of stopping times $\zeta _i$ and events $\mathcal B_i$ for $0\le i\le 2m+1$ in the following way. We let $\zeta _{-1}=0$. For all $0\le i\le m$ we let $\zeta _i$ be the first time $s$ after $4^{i-m}t/2$ that is relaxed with respect to the path $(W_v(s'): s'\in [\zeta _{i-1},s])$ or time $0.6\cdot 4^{i-m}t$ is reached. 

We now define the events $\mathcal B _i$ for $0\le i\le m$. For $i=0$ we let $\mathcal B_0$ be the event that
\begin{enumerate}
    \item 
    We have that $\zeta _0 < 0.6\cdot 4^{-m}t$.
    \item 
    For all $s\le \zeta _1$ we have that $W_v(s)\in B^-$.
    \item 
    We have that $\big\| W_v(\zeta _0) -f (\zeta _0) \big\| \le 2^{-m-5}\sqrt{t}$.
\end{enumerate}
Note that $4^{-m}t= \Theta (\sqrt{\beta })$. By the inductive assumption at time $0.6\cdot 4^{-m}t$ we have that the path $(W(s):s\in [0,0.6\cdot 4^{-m}t])$ is a relaxed path with probability at least $1-\exp (-c\log ^2 \beta )$ and on this event $\zeta _0 < 0.6\cdot 4^{-m}t$. Moreover, before $4^{-m}t$, $W_v$ is a simple continuous time walk and therefore, using a straightforward random walk estimates we have that conditions (2) and (3) hold with probability at least  $ \beta  ^{-C}$. This shows that $\mathbb P (\mathcal B _0 )\ge  \beta ^{-C}$.

Next, for any $1\le i\le m$ we let $\mathcal B _i$ be the event that
\begin{enumerate}
    \item 
    We have $\zeta _i < 0.6\cdot 4^{i-m}t$.
    \item 
    For all $\zeta _{i-1} \le s\le \zeta _i$ we have that $
    \big\| W_v(s) -f(s) \big\| \le 2^{i-m-5}\sqrt{t}$.
\end{enumerate}
We claim that for all $1\le i\le m$ on the event $\bigcap _{i'< i}\mathcal B _{i'}$ we have that $\mathbb P (\mathcal B _i \ |  \ \mathcal F_{\zeta _{i-1}} )\ge c$. To this end, we couple $W_v$ with an independent regenerated walk $\tilde{W}$ starting from $W_v(\zeta _{i-1})$ such that $W_v(s)=\tilde{W}(s-\zeta _{i-1})$ until $W_v$ hits its history. Using the inductive coupling with Brownian motion at time $0.6\cdot 4^{i-m}t-\zeta _{i-1}$ we have that $\tilde{W}$ satisfies 
\begin{equation}\label{eq:event56}
    \mathbb P \big( \forall s\in [\zeta _{i-1},0.6\cdot 4^{i-m}t], \  \| \tilde{W}(s-\zeta _{i-1})-f(s)  \| \le 2^{i-m-5}\sqrt{t} \ | \ \mathcal F _{\zeta _{i-1}} \big) \ge c,
\end{equation}
where in here we used that $\|f'(s) \| \le C2^{m-i-j} \le C /(2^{i-m}\sqrt{t})$ in this time interval. Moreover, it is easy to check that on the event in \eqref{eq:event56} and $\bigcap _{i'<i}\mathcal B _i$, the walk $\tilde{W}$ will not intersect $(W_v(s), s\le \zeta _{i-2})$ before $0.6\cdot 4^{i-m}t$. 
Next, by Lemma~\ref{lem:relaxed4}, with probability at least $1-\beta ^{-c}$, the walk $W_v(s)$ for $s\in [\zeta _{i-1},\zeta _{i-1}+4^n]$ will not hit its history in the interval $[\zeta _{i-2},\zeta_{i-1}]$ since $\zeta _{i-1}$ is relaxed with respect to this history. It follows that with positive probability we have both, $W_v(s)=\tilde{W}(s-\zeta _{i-1})$ for all $s\in [\zeta _{i-1},0.6\cdot 4^{i-m}t]$ and the event in \eqref{eq:event56} holds. Finally, by the inductive hypothesis, the path $(\tilde{W}(s), s\in [\zeta _{i-1},0.6\cdot 4^{i-m}t])$ is relaxed with very high probability and on this event $\zeta _i<0.6\cdot 4^{i-m}t$. We obtain that on $\bigcap _{i'< i}\mathcal B _{i'}$ we have that $\mathbb P (\mathcal B _i \ |  \ \mathcal F_{\zeta _{i-1}} )\ge c$.

We turn to define the stopping times $\zeta _{i}$ and the events $\mathcal B _i$ for $m+1\le  i \le 2m$. We let $\zeta _i$ be the first time $s$ after $t-4^{m-i}t/2$ that is relaxed with respect to $(W_v(s'):s'\in [\zeta _{i-1},s] )$ or time $t-0.3\cdot 4^{m-i}t$ is reached. The event $\mathcal B _i$ for $m+1\le i \le 2m$ is defined to be
\begin{enumerate}
    \item 
    We have that $\zeta _i < t-0.3\cdot 4^{m-i}t$.
    \item 
    For all $\zeta _{i-1} \le s\le \zeta _i$ we have that $
    \big\| W_v(s) -f(s) \big\| \le 2^{m-i-3}\sqrt{t}$.
    \item 
    We have that $\big\| W_v(\zeta _i) -f (\zeta _i) \big\| \le 2^{m-i-5}\sqrt{t}$.
\end{enumerate}
Note that in here $W(s)$ becomes closer to $f(s)$ as $i$ increases and therefore we add condition (3) to the definition. Using the same arguments as in the case $i\le m$ we obtain that $\mathbb P (\mathcal B _i \ |  \ \mathcal F_{\zeta _{i-1}} )\ge c$ on the event $\bigcap _{i'< i}\mathcal B _{i'}$.

Finally, we define $\mathcal B _{2m+1}$ analogously to $\mathcal B _0$. The event $\mathcal B _{2m+1}$ is defined to be 
\begin{enumerate}
    \item 
    For all $\zeta _{2m} \le s \le t$ we have that $W_v(s)\in B^-$.
    \item 
    We have that $W_v(t)=w$.
\end{enumerate}
We claim that on $\bigcup _{i\le 2m}\mathcal B _i$  we have  $\mathbb P (\mathcal B _{2m+1} \ |  \ \mathcal F_{\zeta _{2m}} )\ge \beta ^{-C}$. Indeed, for all $s\le t-\beta $ on the event $\bigcup _{i\le 2m}\mathcal B _i$ we have that $\|W_v(s)-w\| \ge c\sqrt{\beta }$. Moreover, $t-\zeta _{2m}=\Theta (\sqrt{\beta })$ and therefore $\|W_v(\zeta _{2m})-w\|=\Theta (\beta ^{1/4})$. Thus, $(W_v(s), s\in [\zeta _{2m},t])$ will not interact with its history with high probability and we can couple it with a simple continuous time walk. Hence, using straightforward random walk estimates we get that conditions (1) and (2) holds with probability at least $ \beta ^{-C}$.

We obtain that 
\begin{equation}
    \mathbb P \bigg( \bigcap _{i=0 }^{2m+1} \mathcal B _i \bigg) \ge \beta ^{-C}c^{2m} \ge n^{-\gamma +1},
\end{equation}
where the last inequality holds as long as $\gamma $ is sufficiently large using that $m=O(\log n)$. This finishes the proof of the claim since the event of the claim holds on $\bigcap _{i= 0}^{2m+1}\mathcal B _i$.
\end{proof}

It remains to prove Claim~\ref{claim:green}.

\begin{proof}[Proof of Claim~\ref{claim:green}]
Let $W_i'$ be a regenerated walk independent of $\mathcal G(\beta ')$. We can couple $W_i$ and $W_i'$ such that 
\begin{equation}\label{eq:86}
    \mathbb P \big( \forall -4^n \le s \le 0, \ W_i(s)=W_i'(-s) \mid \mathcal G(\beta ') \big) \ge 1-\beta ^{-c}.
\end{equation}
This follows from the same arguments as in the proof of Lemma~\ref{lem:relaxed4}. Indeed, for all  $r\ge \log \beta $, the blocks $p_i+[-r,r]^d$ are not heavy with respect to the union of paths $\bigcup _{\ell \le n^3} \{ W_\ell (s):s\le \tau _\ell \}$ and therefore the walk $W_i'$ will avoid all of them with high probability. The detail are omitted. Next, using the Brownian approximation we obtain that  
\begin{equation}\label{eq:85}
    \mathbb P \big( \forall s\le 4^j, \ W_i'(s)\in u +[0,7\cdot 2^{j-3}]\times [2^{j-2},3\cdot 2^{j-2}] \text{ and } W_i'(4^{j-1})\in B^- \big) \ge c.
\end{equation}
On the intersection of \eqref{eq:86} and \eqref{eq:85} the event of the claim holds.
\end{proof}

\section{Heavy blocks and relaxed times}\label{sec:heavy}

The main result of this section shows that locally, most times are relaxed. In particular, the theorem establishes the inductive step on relaxed times. This result will allow us to concatenate shorter walks and prove the rest of the induction step. 
Recall that  $\epsilon :=1/(300d)$. 

\begin{thm}\label{thm:rel}
    We have that 
    \begin{equation}
        \mathbb P \Big( \forall t\le t_n-t_n^\epsilon ,\ \big| \mathcal T \cap [t,t+t_n^\epsilon ] \big| \ge \big( 0.9+2/n \big) t_n^\epsilon  \Big) \ge 1-e^{-cn^3}.
    \end{equation}
\end{thm}

In order to establish Theorem~\ref{thm:rel} we need the following two subsections. In Subsection~\ref{subsec:heavy} we show that there are no heavy blocks of side length larger than $t_n^{\delta }$. Then, in Subsection~\ref{sec:improvement} we use a multiscale scheme to show that most times are relaxed with respect to the local history of the walk.

It is convenient to rule out short periods where the walk moves too quickly. We thus work before the stopping time $\tau _{{\rm fast}}=\tau _{{\rm fast}}(n^3)$ as defined in \eqref{eq:tau fast}. By Claim~\ref{claim:fast} we have that 
\begin{equation}\label{eq:taufast}
 \mathbb P \big( \tau _{{\rm fast}} \le t_n \big) \le e^{-cn^3}.
\end{equation}
Note that before $\tau _{{\rm fast}}$ we have $\|W(t)\|< 5^n$.

\subsection{Heavy blocks}\label{subsec:heavy}
Recall that $\mathcal B _m$ is the set of blocks of the form $u+[0,2^m)^d$ where $u\in 2^m \mathbb Z ^d$. Let $m_0:=\lfloor \delta \log _2 t_n \rfloor $ and recall that $2^j$ is the scale of the escape algorithm where $j\approx \log \log t_n$. For any $ m \ge m_0$ and any $B\in \mathcal B _m$ we define the sequence of stopping times $s_i=s_i(B)$. We let $s_0$ be the first time the walk enters $B$. Then, we inductively define 
\begin{equation}
    s_{i}:= \inf \big\{ t>s_{i-1} +n4^m : t \text{ is the first time we enter a block }B'\in \mathcal B _j \text{ with } B'\subseteq B \big\}.
\end{equation}

We will show that it is unlikely that many of these stopping times occur before $t_n$. To this end, let $\tau '(B):=s_{n^{\gamma +3}}$ be the first time $n^{\gamma +3}$ of these stopping time occurs, where $\gamma $ is the constant from the definition of $\tau _{\text{esc}}$.

\begin{lem}\label{lem:tau'}
    For all $ m \ge m_0$  and $B\in \mathcal B _m$ we have that 
    \begin{equation}
     \mathbb P \big( \tau '(B) \le t_n\wedge \tau _{\rm{reg}} \big) \le e^{-cn^3}.   
    \end{equation}
\end{lem}

\begin{figure}[htp]
    \centering
     \includegraphics[width=14cm]{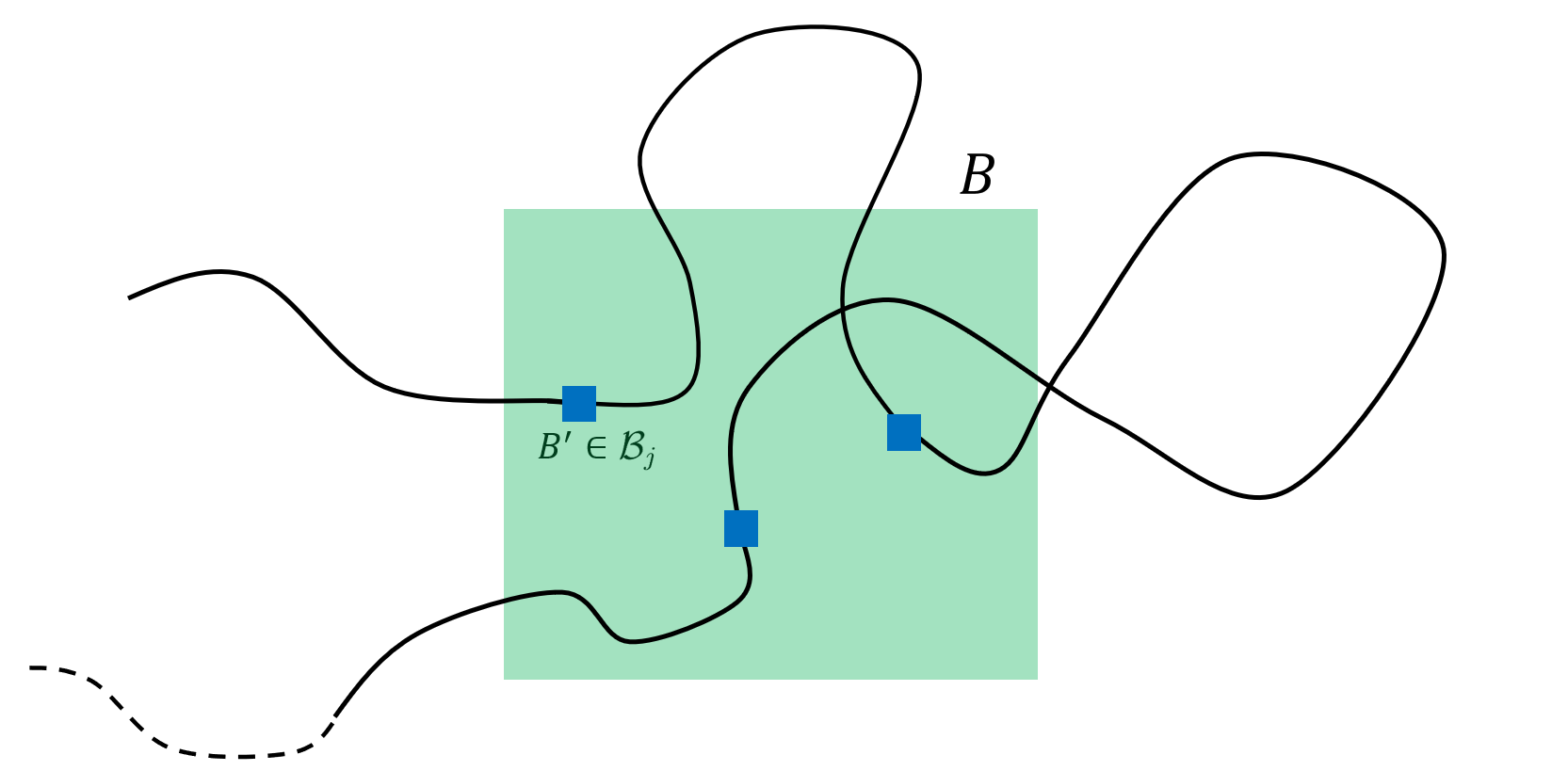}
        \caption{Heavy blocks. We wish to show that a block $B$ of side length $2^m\ge t_n^{\delta }$ is not heavy. Every time we enter a small blue block $B'\in \mathcal B_j$ inside $B$ for the first time, by the escape algorithm, we have a chance of at least $n^{-\gamma }$ to escape fast from $B$ and not return for a long time. If the escape attempt failed and we returned to $B$ after time $n4^m$ then we try again with a new small blue block. With very high probability, this should not happen more than $n^{\gamma +3}$ times which is not enough in order to make $B$ heavy.
    }
    \label{fig:heavy}
\end{figure}

\begin{proof}
    Using the definition of the escape event $\mathcal E (t)$ in \eqref{eq:def of E}, we obtain that on the event $\{ s_{i} < t_n \wedge \tau _{\rm reg} \wedge \tau _{\rm esc} \}$ 
    \begin{equation}
        \mathbb P \big( s_{i+1} \ge s_{i} +4^{n} \mid \mathcal F _{s_{i}} \big) \ge \mathbb P \big( \mathcal E  (s_{i}) \mid \mathcal F _{s_{i}} \big)  \ge n^{-\gamma }.
    \end{equation}
Moreover, the event $\{s_{i+1}\ge s_i+4^n\}$ can happen at most $4$ times before $t_n\le 4^{n+1}$ and therefore 
\begin{equation}
    \mathbb P \big( \tau '(B) =s_{n^{\gamma +3} } \le t_n \wedge \tau _{\rm reg} \wedge \tau _{\rm esc} \big) \le \mathbb P \big( \text{Bin}(n^{\gamma +3}, n^{-\gamma }) \le 4 \big)\le e^{-cn^3}. 
\end{equation}
This finishes the proof of the lemma using Theorem~\ref{thm:escape}.
\end{proof}

Next, define the stopping time 
\begin{equation}
     \tau ':= \min _{m \ge m_0} \min _{B\in \mathcal B _m} \tau ' (B).
\end{equation}
The next corollary follows immediately from Lemma~\ref{lem:tau'}.

\begin{cor}
    We have that
    \begin{equation}
    \mathbb P \big( \tau ' \le t_n \wedge \tau _{\rm reg} \big) \le e^{-cn^3}.
\end{equation}
\end{cor}

\begin{proof}
    On the event $\{ \tau _{\rm fast} > t_n \wedge \tau _{\text{reg}} \}$ all blocks $B$ at distance at least $5^n$ from the origin satisfy $\tau '(B)>s_0(B) > t_n \wedge \tau _{\rm reg}$. Moreover, for all $m\ge n$ and $B\in \mathcal B _m$ we clearly have that $\tau '(B)>s_1(B)>t_n$. Thus, by a union bound over $m_0\le m \le n$ and over blocks $B\in \mathcal B _m$ at distance at most $5^n$ from the origin we obtain
\begin{equation}\label{eq:tau'}
    \mathbb P \big( \tau ' \le t_n \wedge \tau _{\rm reg} \big) \le n5^{dn} e^{-cn^3} +\mathbb P \big( \tau _{\rm fast} \le t_n\wedge \tau _{\rm reg} \big) \le e^{-cn^3},
\end{equation}
where in the last inequality we used \eqref{eq:taufast}.
\end{proof}

\begin{lem}\label{lem:heavy2}
    Before $t_n\wedge \tau _{\rm reg}\wedge \tau _{\rm fast} \wedge  \tau ' $ non of the blocks $B(x,r)$ for $x\in \mathbb Z ^d$ and $r\ge t_n^\delta $ are heavy.
\end{lem}

\begin{proof}
Let $m \ge m_0$ and let $B\in \mathcal B 
 _m$. We will count up the accumulated number of vertices visited in $B$ up to time $t\le t_n\wedge \tau _{\rm reg}\wedge \tau _{\rm fast} \wedge  \tau ' $.

We can split this into 2 parts, the vertices visited in the time intervals $(s_i,s_i+n4^m)$ and those visited in the time intervals $(s_i+n4^m, s_{i+1})$. Before $\tau _{\rm fast}$, we visit at most $n^{C}4^{m}$ vertices in the time interval $(s_i,s_i+n4^m)$ for all $i\ge 0$. Thus, before $\tau '(B)$ we visit at most $n^{C}4^{m}$ vertices in $\bigcup _i (s_i,s_i+n4^m)$.

Next, by the definition of the stopping times $s_i$, for all $t\in (s_i+n4^m,s_{i+1})$ the walk $W(t)$ is either outside of $B$ or in a block $B'\in \mathcal B _j$ with $B'\subseteq B$ that has been visited already in some time interval $(s_{i'},s_{i'}+m4^n)$ for $i'\le i$. By the previous argument the total number of such blocks is at most $n^C4^m$ and hence the total number of vertices in these blocks is at most $n^C4^n$. This shows that the total number of vertices visited in $B$ before $t_n\wedge \tau _{\rm reg}\wedge \tau _{\rm fast} \wedge  \tau '$ is at most $n^C4^m$.

It follows that the blocks $B(x,r)$ for $r\ge t_n^{\delta }$ are not heavy. Indeed, let $m:=\lceil \log _2 r \rceil \ge m_0$ and cover $B(x,r)$ by $2^d$ blocks in $\mathcal B _m$. The number of vertices visited in each one of these blocks before $t_n\wedge \tau _{\rm reg}\wedge \tau _{\rm fast} \wedge  \tau '$ is at most $n^C4^m\le n^Cr^2$. Thus, the number of vertices visited in $B(x,r)$ is at most $n^Cr^2\le r^{5/2}$ and so the block is not heavy. 
\end{proof}

\subsection{The multiscale improvement}\label{sec:improvement}

In this section, we define a key construction in our analysis, the notion of good blocks.  Within a good block all paths either close or are relaxed and do not travel too far.  Furthermore, pairs of paths satisfy the \emph{pair proximity property} and so do not spend much time close to each other. In the definition of good blocks we consider blocks in $\mathcal B _k$ where we fix $k:=\lfloor 4\delta n \rfloor$. Finally, we say that a time $t$ is an integer cyclic time if $t(\!\!\!\!\mod \beta )$ is an integer.

\begin{definition}[Good blocks]
A block $B\in \mathcal B _{k}$ is good if all the following holds for all cyclic walks $W_1,W_2$ of length $t_k':=4^k/k^4$ starting from $u_1,u_2\in N(B,2^{k+2})$ at integer cyclic times $q_1, q_2\le \beta$.
\begin{enumerate}
    \item 
    The walk $W_1$ either closes before time $t_k'$ or is relaxed in time $t_k'$.
    \item
    We have that $$\max_{s\le t_k'} \|W_1(s)\|\le 2^{k-1}.$$
    \item 
    If $W_1$ and $W_2$ do not merge or close then,
    \[
|\fP_k(u_1,u_2,q_1,q_2)| \leq 3.9^k.
    \]
\end{enumerate}
\end{definition}

Note that the event that a block $B$ is good is measurable with respect to the  Poisson clocks inside $N(B,2^{k+3})$. Indeed, before a walk exits $N(B,6\cdot 2^k)$ condition $(2)$ has been already violated and so the block is bad.

The following lemma is in a sense at the heart of our multi-scale analysis.  It says that with very high probability, we will encounter at most $n$ bad blocks in a $5^n$ neighbourhood of the origin.

\begin{lem}\label{claim:multi}
With probability at least $1-e^{-cn^3}$, there are at most $n$ bad blocks $B\in \mathcal B _{k}$ within distance $5^n$ of the origin  \end{lem}

\begin{proof}
Each of the 3 properties required of a good block correspond directly to bounds in the inductive hypothesis. Thus, the probability that a fixed pair of paths fails to satisfy part (3) in the definition of a good block is at most $p_k \le e^{-k^2} \le e^{-cn^2}$. Similarly, the probability that a fixed path fails to satisfy parts (1) or (2) is at most $p_{k'}\le e^{-k'^2}\le e^{-cn^2}$, where $k'$ is the unique integer such that $4^{k'}<t_k'\le 4^{k'+1}$. The number of pairs of starting points and starting times is bounded by $2^{2dk}\beta^2\le e^{Cn}$ and so for a block $B\in \mathcal B _k$ 
\[
\P\big( B \hbox{ is bad} \big) \leq e^{-cn^2}.
\]

The goodness of blocks at distance $20$ are independent.  If there are more than $n$ bad blocks at distance $5^n$ from the origin, then we can find a subset $\mathcal A$ of these blocks with $|\mathcal A| = n/20^d$ such that any pair of blocks are at distance at least $20$. The number of such $\mathcal A$ is bounded by
\[
{(2\cdot 5^n)^d \choose n/20^d} \leq e^{Cn^2}.
\]
Thus, we obtain
\[
\P \big(\#\hbox{Bad blocks} \geq n \big) \le \sum _{\mathcal A} \prod _{B\in \mathcal A}\P\big( B \hbox{ is bad} \big)  \leq e^{Cn^2}e^{-cn^3} \leq e^{-cn^3}.
\]
This finishes the proof of the lemma.
\end{proof}

In the beginning of the induction, when $t_n\le \beta ^C$, a slightly different treatment is required.

\begin{lem}\label{claim:multi2}
Suppose that $t_n^{5\delta } \leq \beta$. With probability at least $1-e^{-cn^3}$, there are no bad blocks $B\in \mathcal B _{k}$ within distance $5^n$ of the origin.  
\end{lem}

\begin{proof}
In this case $t_{k}'\le t_k \le t_n^{5\delta } \le \beta$ and therefore $p_k=e^{-k^3}$ and $p_{k'}=e^{-k'^3}$. Thus, using the same arguments we have the improved probability bound for all $B\in \mathcal B _k$
\[
\P \big( B \hbox{ is bad} \big) \le e^{-cn^3}.  
\]
Taking a union bound over the $2^d 5^{dn}$ blocks we have that
\[
\P \big( \#\hbox{Bad blocks} \geq 1 \big) \le e^{-cn^3},
\]
completing the lemma.
\end{proof}

We say that a block $B\in \mathcal B _k$ is bad by time $t$ if we exposed by time $t$ a cyclic walk of length $t_k'$ starting from $N(B,2^k)$ at an integer time that violates conditions (1) or (2) or a pair of walks violating condition (3). We let $\Lambda_t$ be the union of all bad blocks discovered up to time $t$. Define the stopping time
\begin{equation}\label{eq:def of tau bad}
    \tau_{\text{bad}} := \begin{cases}
    \inf \{t:|\Lambda_t|> n2^k\} & \hbox{if } t_n^{5\delta } > \beta,\\
    \inf \{t:|\Lambda_t|>0\} & \hbox{otherwise}.
\end{cases}
\end{equation}
Before $\tau _{\rm fast}$ the walk stays within distance $5^n$ from the origin and cannot expose bad blocks outside of this region. Thus, by \eqref{eq:taufast}, Lemma~\ref{claim:multi} and Lemma~\ref{claim:multi2} we have that
\begin{equation}\label{eq:tau.bad}
\mathbb P \big(\tau_{\text{bad}}<t_n \wedge \tau _{\text{reg}} \big) \le \mathbb P \big( \tau_{\text{bad}}<t_n \wedge \tau _{\text{reg}} \wedge \tau _{{\rm fast}} \big) +e^{-cn^3} \le e^{-cn^3} .    
\end{equation}

Finally write
\begin{equation}\label{eq:def of hat tau}
\hat{\tau}:= t_n \wedge \tau _{\text{reg}} \wedge  \tau _{\text{fast}} \wedge \tau '\wedge\tau_{\text{bad}}.    
\end{equation}
By \eqref{eq:taufast}, \eqref{eq:tau'} and \eqref{eq:tau.bad} we have that $\mathbb P (\hat{\tau } \le t_n\wedge \tau _{\rm reg})\le e^{-cn^3}$. 

In the following lemma we show that before $\hat{\tau}$ and inside a good block we have a high density of relaxed times.

\begin{lem}\label{l:relaxed.pts}
Suppose that $t_k'\le t<\hat{\tau }$ and that $W(t)\notin \Lambda _t$. Then,
\begin{equation}
\big| \mathcal T \cap  [t-t_k',t] \big|  \ge \big( 0.9+4/n \big) t_k'.
\end{equation}
\end{lem}

\begin{figure}[htp]
    \centering
     \includegraphics[width=14cm]{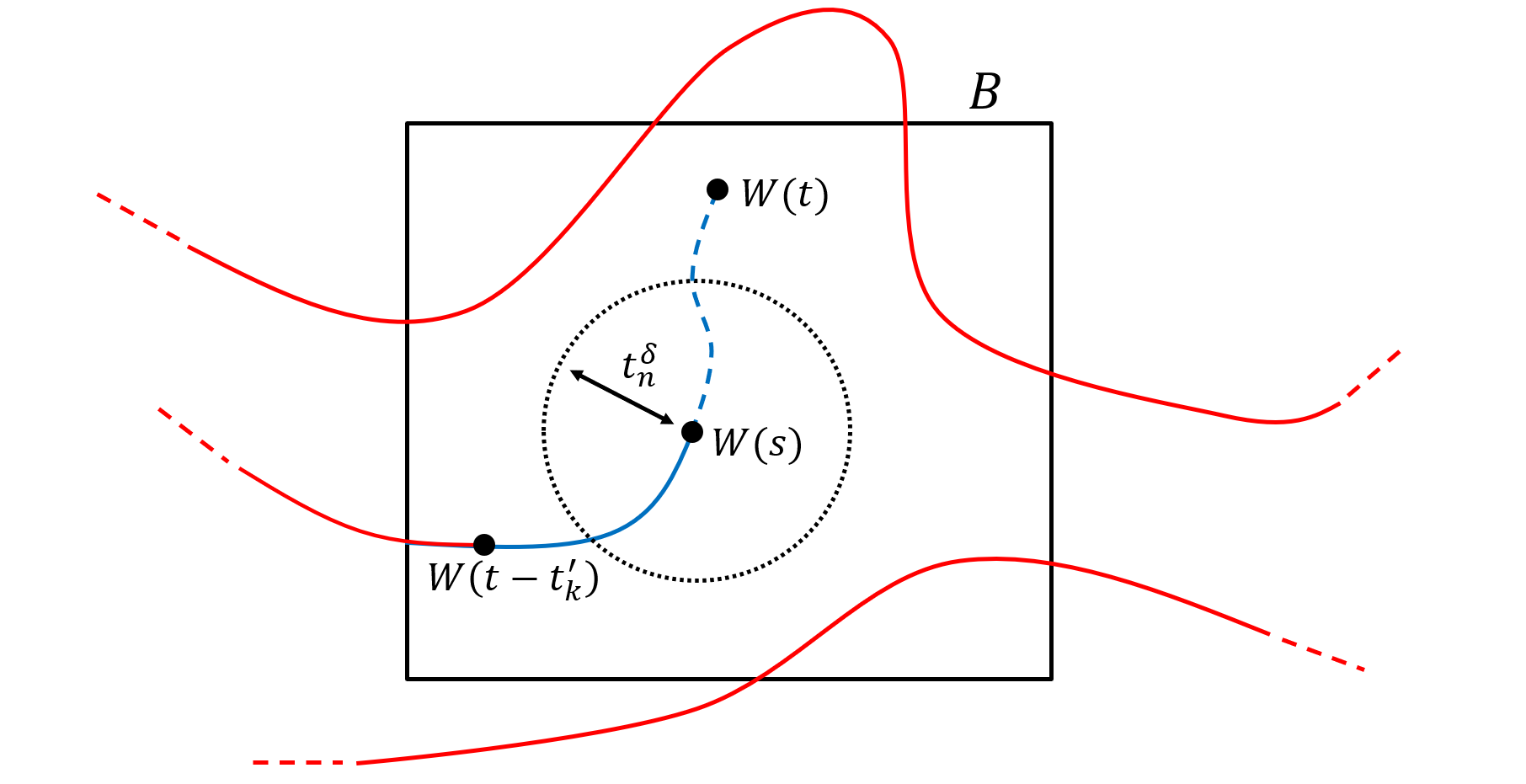}
        \caption{The proof of Lemma~\ref{l:relaxed.pts}. We wish to show that along the blue path, most of the times $s\in [t-t_k',t]$ are relaxed. The blocks $B(W(s),r)$ for $r\ge t_n^\delta $ are not heavy by Lemma~\ref{lem:heavy2} and therefore it suffices to show that the blocks $B(W(s),r)$ for $r\le t_n^\delta $ are not heavy for most times $s$. First, we claim that for most of the times $s$, the earlier history of the walk before $t-t_k'$ (the red paths in the picture) is at distance at least $t_n^\delta $ away from $W(s)$. Indeed, by Claim~\ref{claim:cover}, there are at most $n^C$ paths such that the earlier history is contained in a small neighbourhood of these paths. By the pair property, most of the time the blue path will be far from all of them and therefore far from the entire history before $t-t_k'$. Thus, we only have to show that for most times $s$ the blocks $B(W(s),r)$ for $r\le t_n^\delta $ are not heavy with respect to the blue path. This follows from the fact that $B$ is good and therefore the blue path is relaxed.
    }
    \label{fig:relaxed}
\end{figure}

For the proof of the lemma we will need the following claim.

\begin{claim}\label{claim:cover}
Suppose that $t\le \hat{\tau }$ and let $B\in \mathcal B _k$. If $B$ is a good block at time $t$ then we can construct $\ell \le n^{\gamma +7}$ cyclic integer times $a_1,\dots ,a_\ell  < t-t_k'$  with $W(a_i)\in N(B,2^{k+1})$ such that the walk $(W(s):s\le t)$ inside $N(B,2^k)$ was always at distance at most $n^{41}$ away of one of the cyclic walks $(W(s):a_i\le s \le a_i+t_k')$. That is, we have that
\begin{equation}
N(B,2^k)\cap \{W(s): s\le t \}  \subseteq  \big\{ u : \exists i\le r , \ s\in [a_i,a_i+t_k'], \ \|u-W(s)\|\le n^{41} \big\}.
\end{equation}
\end{claim}

\begin{proof}
The block $N(B,2^k)$ is the union of $3^d$ blocks in $B'\in\mathcal B _k$. Recall the definition of $s_i=s_i(B')$ from section~\ref{subsec:heavy}. Since $t\le \tau '$, the set of times
$$T:= [0,t-2] \cap \bigcup _{i} \big( s_i(B'),s_i(B')+n4^k \big) $$ 
is contained in at most $n^{\gamma }$ intervals of length $n4^k$. Thus, this set can be covered by at most $n^{\gamma +7}$ intervals $(a_m,a_m+t_k')$ such that $a_m<t-t_k'$ is a cyclic integer. Moreover, by the definition of the stopping times $s_i(B')$ and $\tau _{\text{fast}}$, any $s\le t $ such that $W(s)\in B'$ satisfies $\|W(s)-W(s')\|\le n^{41}$ for some $s'\in T$.

Next, note that if $W(a_m)\notin N(B,2^{k+1})$ then the walk $(W(s):a_m\le s\le a_m+t_k')$ do not enter $N(B,2^{k})$. Indeed, otherwise, using that $a_m+t_k'< \tau _{\rm fast}$, there exists some cyclic integer time $x\in [a_m,a_m+t_k']$ and $x<y<a_m+t_k'$ such that  $3\cdot 2^{k-1}\le d(W(x),B) \le 2^{k+1}$ and $W(y)\in N(B,2^{k})$. This is a contradiction since $B$ is good and the cyclic walk $(W(z):x<z<x+t_k')$ cannot travel so fast. The claim now follows by taking all the $a_m$ for which $W(a_m)\in N(B,2^k)$ obtained from the $3^d$ blocks $B'$ inside $N(B,2^k)$.
\end{proof}

We can now prove Lemma~\ref{l:relaxed.pts}.

\begin{proof}[Proof of Lemma~\ref{l:relaxed.pts}]
For simplicity, assume that $t':=t-t_k'$ is a cyclic integer and let $B\in \mathcal B _k$ be the block containing $W(t)$. By the assumption $B$ is good. Define the sets
\begin{equation}
    U_1 := \big\{ s\in [t',t] : s \text{ is relaxed with respect to the path } (W(x) : x\in [t',s]) \big\}
\end{equation}
and
\begin{equation}
    U_2 : = \big\{ s\in [t',t] : \min _{s'<t'} \| W(s)-W(s') \| > t_n^\delta  \big\}.
\end{equation}
Since $B$ is good and using the same arguments as in the proof of Claim~\ref{claim:cover}, we have that $W(t')\in N(B,2^k)$ and therefore path $(W(s), s\in [t',t])$ is a relaxed path. It follows that $|U_1|\ge \big( 0.9+1/k' \big) t_k' \ge \big( 0.9+5/n \big) t_k'$.

We turn to show that $U_2$ is large. By Claim~\ref{claim:cover} with $t'$ as $t$, there are $\ell \le n^{\gamma +7}$ cyclic integers $a_1\dots ,a_\ell <t'-t_k'$ with $W(a_i)\in N(B,2^{k+1})$ such that 
\begin{equation}\label{eq:cover}
N(B,2^k) \cap \{W(s) : s\le t' \}  \subseteq 
    \big\{ u : \exists i\le r , \ s\in [a_i,a_i+t_k'], \ \|u-W(s)\|\le n^{41} \big\}
\end{equation}

It is clear that the cyclic walks $(W(s):s\in [t',t])$ and $(W(s): s\in [a_i,a_i+t_k'])$ do not close or merge and therefore, using condition (3) in the definition of a good block we have that
\begin{equation}\label{eq:pair}
    \big|  \big\{ s\in [t',t] : \exists s'\in [a_i,a_i+t_k'] \text{ such that } \|W(s)-W(s')\|\le 1.9^k \big\} \big| \le 3.9^k.
\end{equation}
It follows from \eqref{eq:pair} and \eqref{eq:cover} that 
\[
|U_2|\geq t_k'- n^{\gamma +7} \cdot 3.9^k,
\]
where in here we also used that $t_n^\delta <\!<1.9^k$ and that the walk $(W(s):s\in [t',t])$ is contained in $N(B,0.6\cdot 2^{k})$ by the same arguments as in Claim~\ref{claim:cover}. Thus,
\[
|U_1\cap U_2| \ge \big(0.9 +4/n\big) t_k'.
\]
It remains to prove that every $s\in U _1 \cap U_2 $ is a relaxed time. First, we claim that for all $0\le s'\le s-\beta ^{3\delta }$ we have that $\|W(s')-W(s)\| \ge \beta ^\delta  $. Indeed, since $s\in U_1$ this holds for all $t'\le s'\le s-\beta ^{3\delta }$ and since $s\in U_2$, for all $s'\le t'$ we have $\|W(s')-W(s)\| \ge t_n^\delta \ge  \beta ^\delta$.

Next, we show that for all $r\ge \beta ^\delta $ the blocks $B(W(s),r)$ are not heavy. If $r\ge t_n^\delta $ then $B(W(s),r)$ is not heavy by Lemma~\ref{lem:heavy2}, using that $s\le t \le \hat{\tau }$. Finally, let $\beta ^\delta \le r\le t_n^{\delta }$. Since $s\in U_1$ the block $B(W(s),r)$ is not heavy with respect to the path $(W(s):s\in [t',t])$. Since $s\in U_2$, the earlier history is disjoint from $B(W(s),r)$ and so this block is not heavy.
\end{proof}

\begin{cor}\label{cor:relaxed.pts}
Suppose that $t<\hat{\tau }$. Then we have that
\begin{equation}
\big| \mathcal T \cap  [t-t_n^{\epsilon /2},t]  \big| \ge \big( 0.9+3/n \big)t_n^{\epsilon /2}.
\end{equation}
\end{cor}
\begin{proof}
If $t_n^{5\delta } \leq \beta $ then $|\Lambda_t|=0$ by the definition of $\tau_{\rm bad}$ and therefore, the corollary follows immediately from Lemma~\ref{l:relaxed.pts}.  Next, assume that $t_n^{5\delta } \geq \beta $ and consider the set 
\begin{equation}
    U :=\big\{ s\in [t-t_n^{\epsilon /2},t] : W(s)\notin \Lambda _s \big\}.
\end{equation}
Before $\tau _{\rm reg}$ we spend at most $\beta $ time in each vertex. Moreover, before $\tau _{\rm bad}$ there are at most $n$ bad blocks and therefore 
$$|U| \le \beta |\Lambda _t| \le \beta n 2^{dk} \le t_n^{5\delta }n2^{dk}\le Cnt_n^{3\delta d} \le  t_n^{\epsilon /3},$$
where in the fourth inequality we used that $d\ge 5$ and that $k\le 4\delta n$ and in the fifth inequality we used that $\delta =\epsilon /(10d)$. Finally, by Lemma~\ref{l:relaxed.pts}, for each time $s\in [t-t_n^{\epsilon /2},t] \setminus U$ there is a density of at least $(0.9+4/n)$ of relaxed times in $[s-t_k',s]$. The corollary follows from this. 
\end{proof}

\begin{remark}
    The last proof is the reason we have to modify the argument slightly in the beginning of the induction, when $t_n^{5\delta }\le \beta $. In the proof we used the straightforward bound that a cyclic walk stays at most $\beta $ time in each vertex before closing and therefore the walk doesn't spend much time in the bad blocks. However, when $\beta \ge t_n^{5\delta } $ this bound is too large in terms of $t_n$. We overcome this issue by obtaining better probability bounds when $t_n\le \beta $. These bounds are used to show that with sufficiently high probability there are no bad blocks at all when $t_n^{5\delta }\le \beta $.  
\end{remark}

We can now prove Theorem~\ref{thm:rel}.  Define the event $ \cG(1):=\big\{ \hat{\tau } \ge \tau _{\text{reg}} \wedge t_n  \big\}$ and note that $\mathcal P (\mathcal G (1))\ge 1-e^{-cn^3}$. This only involves the behaviour of the process up to the first regeneration time.  We let $\cG(i)$ be the analogous event for the $i$th iteration of the process between the $i-1$th and $i$th regeneration times. We clearly have that $\mathbb P (\mathcal G (i))\ge 1-e^{-cn^3}$ for all $i\ge 0$. Next, we let
\[
\cG:=\Big\{\fR(t_n) \leq n^3, \cG\big(1\big),\ldots,\cG\big(\fR(t_n)+1\big)\Big\}.
\]

\begin{lem}\label{l:cG}
We have $P(\cG ) \geq 1-e^{-cn^3}$.
\end{lem}
\begin{proof}
By the same arguments as in the proof of Theorem~\ref{t:main.thm}, the number of regenerations between $[(i-1)4^n,i4^n]$ is stochastically dominated by a geometric random variable and so
\[
\P\big( \fR(t_n)\geq n^3 \big) \leq e^{-cn^3}.
\]
Taking a union bound completes the lemma.
\end{proof}

\begin{proof}[Proof of Theorem~\ref{thm:rel}]
By Lemma~\ref{l:cG}, it suffices to show that on the event $\cG$ we have for all $t\le t_n-t_n^{\epsilon }$,
\begin{equation}\label{eq:most.relax*}
\big| \mathcal T \cap [t,t+t_n^{\epsilon }]  \big| \ge \big( 0.9+2/n \big) t_n^{\epsilon } .
\end{equation}
Note that as stated, Lemma~\ref{cor:relaxed.pts} applies only up to the first regeneration time. However, we can apply it between any two regeneration times. Thus, for any $1\leq \ell \leq  \lfloor t_n^{\epsilon /2} \rfloor $ such that there is no regeneration time in $[t+(\ell-1)t_n^{\epsilon /2},t+\ell t_n ^{\epsilon /2}]$ we have that
\[
\big| \mathcal T \cap  [t+(\ell-1)t_n^{\epsilon /2} ,t+\ell t_n ^{\epsilon /2}] \big| \ge \big( 0.9+3/n \big) t_n^{\epsilon /2}.
\]
On the event $\cG$ there are at most $n^3$ regeneration times so
\begin{equation}
\big| \mathcal T \cap [t,t+t_n^{\epsilon }]  \big| 
\ge \big( \lfloor t_n^{\epsilon /2}  \rfloor  - n^3\big) \big( 0.9+3/n \big) t_n^{\epsilon /2} \ge \big( 0.9+2/n \big) t_n^{\epsilon }.
\end{equation}
Hence equation~\eqref{eq:most.relax*} holds which completes the proof.
\end{proof}

\section{Concatenation embedding and transition probabilities}\label{sec:concatenation}

We will often use the following lemma in order to concatenate independent shorter walks, for which the induction hypothesis holds, into one long walk. Throughout this section $W$ will be a regenerated walk running for time $t_n$. For walks $W_1$ and $W_2$ of length $s_1$ and $s_2$ respectively we define the concatenation $\tilde{W}$ of $W_1$ and $W_2$ by 
\begin{equation}
    \tilde{W}(s):=\begin{cases}
    \quad \quad \quad  W_1(s) \quad \quad \quad \quad \quad \quad \ \  s\le s_1
\\
W_1(s_1)+W_2(s-s_1) \quad s_1\le s \le s_1+s_2
    \end{cases}.
\end{equation}
We similarly define the concatenation of more than two walks.

\begin{lem}\label{lem:conca}
Let $s_1,s_2$ such that $s_1+s_2 \le t_n$ and let $W_1$ and $W_2$ be independent regenerated walks running for times $s_1$ and $s_2$ respectively. Let $\tilde{W}$ be the concatenation of $W_1$ and $W_2$. Then, with the natural coupling of $W$ and $\tilde{W}$ we have
\begin{equation}
\mathbb P \Big( \max _{s\le s_1+s_2} \big\| W(s)-\tilde{W}(s) \big\| \ge t_n^{2\epsilon} \Big) \le  e^{-cn^3}.    
\end{equation}
\end{lem}

\begin{proof}
In order to couple $W$ and $\tilde{W}$ we simply drive $W$ and $\tilde{W}$ using the same continuous time walk. According to this coupling we have $W(s)=W_1(s)$ for all $s\le s_1$. The coupling between the two walks deteriorate when they are forced by different histories. Every time that happens, the idea would be to use Theorem~\ref{thm:rel} in order to quickly find a time that is relaxed with respect to both, $W$ and $\Tilde{W}$. From this time onward, there is a positive chance that the two walks will be coupled perfectly for a long time.  

For convenience we extend $W$ and $W_2$ to be infinite regenerated walks (so that $W$ and $\tilde{W}$ are defined for all $t>0$). Next, recall the definition of $\tau _{\text{hit}}(t)$ in \eqref{eq:hit}. We similarly define $\tilde{\tau }_{\text{hit}} (t)$ to be the first time $\tilde{W}$ hits the relevant history after the concatenation time $s_1$. That is, for $t>s_1$ we let 
\begin{equation}
    \tilde{\tau } _{\text{hit}}(t) :=\inf \big\{ s>t : \exists s'\in s+\beta \mathbb Z, \ \alpha (t)\vee s_1 \le s'<t \text{ such that } \ \| W(s)-W(s') \| \le 1  \big\}.
\end{equation}
Next, let $\Tilde{\mathcal T }$ be the set of times of the form $s_1+s$ where $s$ is a relaxed time with respect to $W_2$. By Lemma~\ref{lem:relaxed4} we have that 
\begin{equation}\label{eq:tau tilde}
\mathbb P \big( \tilde{\tau } _{\text{hit}}(t) \ge t+4^n \mid \mathcal F _t \big) \ge 1-\beta ^{-c} \quad \text{for all} \quad  t\in  \Tilde{\mathcal T}.  
\end{equation}
Next, we define a sequence of stopping times $\zeta _1,\zeta _2 \dots $ and $\zeta _0', \zeta _1', \zeta _2 '\dots $ in the following way. We let $\zeta _0':=s_1$. For all $i\ge 1$ we let
\begin{equation}
    \zeta _i:=\inf \big\{ t>\zeta _{i-1}' : t\in \mathcal T \cap \Tilde{\mathcal T}   \big\} \quad \text{and} \quad \zeta _i': = \tau _{\text{hit}}(\zeta _i) \wedge \tilde{\tau} _{\text{hit}}(\zeta _i).
\end{equation}
By \eqref{eq:tau tilde} and the analogous statement without the tilde we have that
$$\mathbb P \big( \zeta_{i}'\geq \zeta_i+4^n \mid \mathcal F _{\zeta _i} \big) \ge 1/2 \quad \text{for all} \quad i\ge 1.$$
As before, the event $\zeta_{i}'\geq \zeta_i+4^n$ cannot happen more than $4$ times before $t_n\le 4^{n+1}$ and therefore
\begin{equation}\label{eq:n^4}
    \mathbb P \big( |\{i: \zeta _i< t_n\}|\le n^3 \big) \ge \mathbb P \big( \text{Bin}(n^3,1/2) \ge 5 \big)\ge 1-e^{-cn^3}.
\end{equation}
Moreover, $W$ and $\tilde{W}$ are coupled perfectly in the intervals $(\zeta _i, \zeta _i')$. Indeed, in this interval the walks are only ``forced" by the history after time $\zeta _i$ which is identical for $W$ and $\tilde{W}$. Thus, for all $s\le t_n$ we can write
\begin{equation}
    W(s)-\tilde{W}(s)=\sum _{i=1}^{\infty } \big( W(s\wedge \zeta _i)-W(s\wedge \zeta _{i-1}') \big) - \big( \tilde{W}(s\wedge \zeta _i)-\tilde{W}(s\wedge \zeta _{i-1}') \big).
\end{equation}
Let $\mathcal A$ be the event that for all $s_1\le t\le t_n-t_n^\epsilon $ we have 
\begin{equation}
    \big| \mathcal T \cap [t,t+t_n^\epsilon ] \big| \ge \big( 0.9+2/n \big) t_n^\epsilon \quad \text{and} \quad  \big| \tilde{\mathcal T} \cap [t,t+t_n^\epsilon ] \big| \ge \big( 0.9+2/n \big) t_n^\epsilon. 
\end{equation}
By Theorem~\ref{thm:rel} applied twice for $W$ and $W_2$ we obtain that $\mathbb P (\mathcal A )\ge 1-e^{-cn^3}$. On the event $\mathcal A$ we have for all $i\ge 1 $ that $s\wedge \zeta _i-s\wedge \zeta _{i-1}' \le t_n^{\epsilon }$. Finally, let $\mathcal B$ be the event that $\tau _{\rm fast}(n^3)\ge t_n$ and $\tilde{\tau} _{\rm fast}(n^3)\ge t_n$, where $\tilde{\tau} _{\rm fast}(n^3)$ is defined by \eqref{eq:tau fast} for the walk $\tilde{W}$. By Claim~\ref{claim:fast} applied for the walks $W,W_1$ and $W_2$ we have that $\mathbb P (\mathcal B )\ge 1-e^{-cn^3}$.

On the event $\mathcal A \cap \mathcal B $ and the event in \eqref{eq:n^4} we have that
\begin{equation}
    \|W(s)-\tilde{W}(s)\| \le 2n^3 \sum _{i=1}^{n^3} s\wedge \zeta _{i}- s\wedge \zeta _{i-1} ' \le 2n^{6}t_n^{\epsilon } \le  t_n^{2\epsilon }, 
\end{equation}
which completes the proof of the lemma.
\end{proof}

The next Corollary follows immediately by induction.
\begin{cor}\label{cor:conca}
Let $s_1,\dots ,s_\ell >0$ such that $s_1+\cdots +s_\ell\le t_n$. Let $W_1,\dots ,W_\ell$ be independent regenerated cyclic walks running for times $s_1,\dots ,s_\ell$ respectively. Let $\tilde{W}$ be the concatenation of the walks $W_1,\dots ,W_\ell$. Then, in the natural coupling of $W$ and $\tilde{W}$ we have 
\begin{equation}
\mathbb P \Big( \sup _{s\le s_1+\cdots +s_\ell} \big\| W(s)-\tilde{W}(s) \big\| \ge \ell t_n^{2\epsilon} \Big) \le e^{-cn^3}.
\end{equation}
\end{cor}

\iffalse
\begin{remark}\label{rem:conca}
We note that a slightly stronger statement follows from the proofs of Lemma~\ref{lem:conca} and Corollary~\ref{cor:conca}. This stronger statement will be needed in order to prove the relaxed path property at time $t_n$. As in Corollary~\ref{cor:conca}, we let $\tilde{W}$ be the concatenation of independent walks $W_1,\dots ,W_k$ of lengths $s_1,\dots ,s_\ell$ respectively. Then, in the natural coupling the following holds with probability at least $1-2\ell e^{-8n^2}$. There are times $a_1,\dots a_{\ell n^5}$ such that
\begin{enumerate}
    \item 
    The walks $W$ and $\tilde{W}$ are coupled perfectly outside of $\bigcup _i [a _i, a _i+t_n^{\epsilon }]$. That is, for all $a _i+t_n^{2\epsilon } <s<t<a _{i+1}$ we have that $W(t)-W(s)=\tilde{W}(t)-\tilde{W}(s)$.
    \item 
    No concatenation times of $\tilde{W}$ outside of $\bigcup _i [a _i, a _i+t_n^{2\epsilon }]$. That is, for all $j\le k$ we have that $s_1+\cdots +s_j\notin \bigcup _i [a _i, a _i+t_n^{2\epsilon }]$ .
    \item
    No regeneration times of $\tilde{W}$ outside of $\bigcup _i [a _i, a _i+t_n^{2\epsilon }]$. That is, for all $j$ and a regeneration time $t\le s_j$ of $W_j$ we have that $t+s_1+\cdots +s_{j-1} \notin \bigcup _i [a _i, a _i+t_n^{2\epsilon }]$.
\end{enumerate}
\end{remark}
\fi

\subsection{Traveling far is unlikely}
In this section we concatenate independent walks and use concentration results for sum of independent variables to say that the concatenated walk cannot travel far. Throughout this subsection we fix $\sqrt{t_n} \le t\le t_n$ and bound the fluctuations of the regenerated walk up to time $t$. We do this in three steps, first we use Corollary~\ref{cor:far} to say that large deviations are unlikely in the scale of $t^{1/3}$. Then, we use this to show that intermediate deviations are unlikely in the scale of $t^{2/3}$. Finally, we use this to show that small deviations are unlikely in the scale of $t$. 

\begin{lem}\label{lem:2/3}
For all $\sqrt{t_n} \le t\le t_n$ we have that 
\begin{equation}
    \mathbb P \Big( \max _{s\le t^{2/3}} \|W(s)\| \ge t^{2/5} \Big) \le e^{-cn^3}. 
\end{equation}
\end{lem}

\begin{proof}
For clarity we ignore the floor notations in this proof and assume that $t^{1/3}$ is an integer. Let $k= t^{1/3} $ and let $W_1,\dots ,W_k$ be independent regenerated walks running for time $t^{1/3}$. Let $\tilde{W}$ be the concatenation of $W_1,\dots ,W_k$. By Corollary~\ref{cor:conca} there is a coupling of $W$ and $\tilde{W}$ such that 
\begin{equation}\label{eq:conca2}
   \mathbb P \Big( \max _{s\le t^{2/3}} \|W(s)-\tilde{W}(s)\| \ge t^{1/3+4\epsilon } \Big) \le e^{-cn^3}.
\end{equation}
It is thus suffices to show that $\tilde{W}$ do not travel far. To this end, define the events $\mathcal A _i :=\big\{ \sup_{s\le t^{1/3}} \|W_i(s)\|\le n^3 t^{1/3} \big\}$ and the random variables $X_i:=\mathds 1_ {\mathcal A _i}W_i(t^{1/3})$.
By Corollary~\ref{cor:far} we have that $\mathbb P (\mathcal A_i) \ge 1- e^{-cn^3}$.

Next, we claim that $\mathbb E [\|X_i\|^2] \le 2dt^{1/3}$. Let $\mathcal B _i$ be the event that the inductive coupling of $W_i$ with its Brownian motion $B_i$ was successful. That is 
\begin{equation}
    \mathcal B _i: = \Big\{ \max _{s\le t^{1/3}} | W_i(s)-\sigma B_i(s)| \ge t^{2/15} \Big\}.
\end{equation}
where $\sigma ^2 \le 3$ and $B_i$ are from the inductive assumption. Using the inductive hypothesis for time $t^{1/3}$ we have that $\mathbb P (\mathcal B _i) \ge 1- e^{-cn^2}$. Hence, 
\begin{align*}
    \mathbb E [\|X_i\|^2] &\le \mathbb E \big[ \|X_i\|^2 \mathds 1 _{\mathcal B_i} \big] + (n^5 t^{1/3})^2 \cdot  \mathbb P (\mathcal B _i^c ) \\
    &\le 2\mathbb E \big[ \|\sigma B_i(t^{1/3})\|^2  \big] + 2\mathbb E \big[ \|W_i(t^{1/3})-\sigma B_i(t^{1/3})\|^2 \mathds 1_{\mathcal A_i \cap \mathcal B_i} \big] +Ce^{-cn^2} \le 6dt^{1/3}.
\end{align*}
By symmetry $\E X_i =0$ so we can now use Freedman's inequality \cite[Theorem~18]{chung2006concentration} 
to obtain
\begin{equation}
    \mathbb P \Big( \max _{j\le k } \Big\| \sum _{i=1}^j X_j \Big\| \ge t ^{2/5}/3   \Big) \leq 2d\exp \Big( \frac{-t^{4/5}/18}{ 6 d t^{2/3} + n^{5}t^{1/3}t^{2/5}/9 } \Big) \le e^{-cn^3}.
\end{equation}
Note that the Theorem~18 in \cite{chung2006concentration} deals with the one dimensional case. We therefore need to apply the theorem separately to each coordinate and then union bound over the coordinates.
Using that $\mathbb P \big( \bigcap_{i=1}^{k}\mathcal A _i \big) \ge 1- e^{-cn^3}$ and that on this event $W_i(t^{1/3})=X_i$ we obtain
\begin{equation}\label{eq:tilde W}
     \mathbb P \Big(  \max _{s\le t^{2/3}} \|\tilde{W}(s)\| \ge t^{2/5}/2  \Big) \le e^{-cn^3}.
\end{equation}
The lemma follows from \eqref{eq:conca2} and \eqref{eq:tilde W}.
\end{proof}

\begin{lem}\label{l:too.far}
For all $\sqrt{t_n} \le t \le t_n$ we have that 
\begin{equation}
    \mathbb P \Big( \max _{s\le t} \|W(s)\| \ge n^2 \sqrt{t} \Big) \le e^{-cn^3}. 
\end{equation}
\end{lem}

\begin{proof}
The proof is almost identical to the proof of Lemma~\ref{lem:2/3} and some of the details are omitted. Let $k:=t^{1/3}$ and let $W_1,\dots ,W_k$ be independent regenerated walks running for time $t^{2/3}$. Let $\tilde{W}$ be the concatenation of $W_1,\dots ,W_k$. By Corollary~\ref{cor:conca} we have
\begin{equation}\label{eq:conca33}
   \mathbb P \Big( \max _{s\le t} \|W(s)-\tilde{W}(s)\| \ge t^{1/3+4\epsilon } \Big) \le  e^{-cn^3}.
\end{equation}

Next, we similarly define $\mathcal A _i := \big\{\sup_{s\le t^{2/3}} \|W_i(s)\| \le t^{2/5} \big\}$ and $X_i:=\mathds 1 _{\mathcal A_i} W(t^{2/3})$. By Lemma~\ref{lem:2/3} we have that $\mathbb P (\mathcal A _i)\ge 1-e^{-cn^3}$. Using the same arguments we have that $\mathbb E [\|X_i\|^2] \le 2dt^{2/3}$ and therefore by Freedman's inequality
\begin{equation}
    \mathbb P \Big( \max _{j\le k } \Big\| \sum _{i=1}^j X_j \Big\| \ge n^2\sqrt{t}/3   \Big) \leq 2d\exp \Big( \frac{-n^4 t/18}{ 2 d t + t^{2/5}n^2\sqrt{t}/9 } \Big) \le e^{-cn^3}.
\end{equation}
This finishes the proof of the lemma using the same arguments as before.
\end{proof}

\subsection{Coupling with a Brownian motion}\label{sec:KMT}

To couple the regenerated walk with Brownian motion we make use of Theorem~1.3 due to Zaitsev \cite{zaitsev1998multidimensional} that generalizes the classical KMT~\cite{komlos1975approximation} result to higher dimensions. We present a special case of the Theorem. 

\begin{thm}[Zaitsev]\label{thm:KMT}
Let $X_1,\dots ,X_\ell$ be IID random variables in $\mathbb R ^d$ and let $S_j:=\sum _{i=1}^j X_i$. Suppose that
\begin{equation}
\|X_i\| \le M \quad \text{and} \quad  {\rm Cov}(X_i)=I. 
\end{equation}
Then, there is a coupling of the variables $X_1,\dots ,X_\ell$ and a Brownian motion $B_t$ such that 
\begin{equation}
    \mathbb P \Big( \max _{j\le \ell} \big\|S_j -B(j)\big\|  \ge M\log ^4 \ell \Big) \le \exp \big( -c_d \log ^4 \ell \big).
\end{equation}
The constant $C_d$ depends on the dimension $d$ but not on the distribution of the $X_i$.
\end{thm}
Using this embedding theorem we will establish the coupling between $W(t)$ and Brownian motion.
\begin{lem}\label{cor:kmt}
There is a constant $\sigma =\sigma (t_n)$ with 
\begin{equation}\label{eq:sigma}
    |\sigma -\sqrt{2} |\le \max\{3\beta ^{-1/16} -t_n^{-1/16},0\}
\end{equation}
such that the following holds. There is a coupling of $W$ and a Brownian motion $B(t)$ such that 
\begin{equation}
    \mathbb P \Big( \max _{s\le t_n} \big\| W(s)-\sigma B(s) \big\| \ge t_n^{2/5} \Big) \le e^{-cn^3}.
\end{equation}
\end{lem}

\begin{proof}
Once again, we ignore the floor notation throughout this proof. Let $\ell:=t_n^{1/3}$ and let $W_1,\dots ,W_\ell$ be independent regenerated cyclic walks running for time $t_n^{2/3}$. As usual, $\tilde{W}$ is the concatenation of $W_1,\dots ,W_\ell$ and we have a coupling such that 
\begin{equation}\label{eq:conca3}
   \mathbb P \Big( \max _{s\le t_n} \|W(s)-\tilde{W}(s)\| \le t_n^{1/3+2\epsilon } \Big) \le e^{-cn^3}.
\end{equation}

Define the events $\mathcal A _i:=\big\{ \sup_{s\le t_n^{2/3}} \|W_i(s)\| \le n^2t_n^{1/3} \big\}$ and note that by Lemma~\ref{l:too.far} we have $\mathbb P (\mathcal A _i) \ge  1-e^{-cn^3}$. Define the random variables $Y_i:=t_n^{-1/3}\mathds 1 _{\mathcal A _i}  W_i(t_n^{2/3})$ and let $\sigma =\sigma (t_n)$ be the standard deviation of the first coordinate of $Y_i$.  

First, we will establish that \eqref{eq:sigma} holds. Set $\sigma '=\sigma (t_n^{2/3})$ be the constant from the inductive coupling with Brownian motion at time $t_n^{2/3}$ and so 
\begin{equation}\label{eq:sigma'}
    |\sigma '-\sqrt{2}| \le \max \big( 3\beta ^{-1/16}- t_n^{-2/60}, 0\big).
\end{equation}
Next, let $\mathcal B_i$ be the event that the coupling of $W_i$ with its Brownian motion $B_i$ was successful. We have that $\mathbb P (\mathcal B _i) \ge 1-e^{-cn^2}$.  Denote by $(Y_i)_1$ the first coordinate of $Y_i$ and by $\big(B_i(t_n^{2/3}) \big) _1$ the first coordinate of $B_i(t_n^{2/3})$. By the symmetry of the lattice $\mathbb Z ^d$ we have that $\mathbb E [( Y_i )_1 ]=0$ and therefore $\sigma ^2 =\mathbb E [ (Y_i)_1^2 ]$. Thus, using the triangle inequality in the $L^2(\Omega )$ space and Cauchy Schwarz inequality we obtain 
\begin{equation}
\begin{split}
    t_n^{2/3}|\sigma -&\sigma '|^2 \le \mathbb E \big(  t_n^{1/3}(Y_i)_1- \sigma ' \big(B_i(t_n^{2/3})\big)_1 \big)^2 \\
    &\le \mathbb E \big[ \mathds 1_{\mathcal A _i \cap \mathcal B _i} \big\| W_i(t_n^{2/3})-\sigma 'B_i(t_n^{2/3}) \big\|^2 \big] + 2\mathbb E \big[ \mathds 1_{\mathcal A _i^c \cup \mathcal B _i^c} \big( t_n^{1/3}\|Y_i\|+\|B_i(t_n^{2/3})\| \big)^2 \big]\\
    &\le C\big( t_n^{2/3} \big)^{4/5} + \sqrt{\mathbb P (\mathcal A _i^c \cup \mathcal B _i ^c) \cdot \mathbb E \big[ \big( n^2 t_n^{1/3}+ \| B_i(t_n^{2/3}) \|  \big)^4 \big]} \le Ct_n^{8/15}+Ce^{-cn^2}.
\end{split}
\end{equation}
We obtain that $|\sigma -\sigma '|\le t_n^{-1/16}$ and so
\[
|\sigma-\sqrt{2}|\leq |\sigma '-\sqrt{2}| + |\sigma -\sigma '| \leq \max \big( 3\beta ^{-1/16}- t_n^{-2/60}, 0\big) + t_n^{-1/16}.
\]
This finishes the proof of \eqref{eq:sigma} using \eqref{eq:sigma'} in both cases, when the maximum in \eqref{eq:sigma'} is zero and when it's positive.

Next, let $X_i:=\sigma ^{-1}Y_i$. By the symmetries of $\mathbb Z ^d$ the covarice between different coordinates of $X_i$ is $0$ and therefore its covariance matrix is $I$. Moreover, we have that $\|X_i\| \le 2n^2$. Thus, by Theorem~\ref{thm:KMT}, there is a Brownian motion $B'$ such that 
\begin{equation}
    \mathbb P \Big( \max _{j\le k} \big\| S_j-B'(j) \big\| \ge n^7 \Big) \le e^{-cn^3},
\end{equation}
where $S_j:=\sum _{i=1}^j X_i$. Moreover, on the event $\mathcal A :=\bigcap _{i\le k} \mathcal A _i$, we have that $S_j=\sigma ^{-1}t_n^{-1/3}\tilde{W}(jt_n^{2/3}) $ for all $j\le k$ and therefore 
\begin{equation}
    \mathbb P \Big( \max _{j\le k} \big\| \tilde{W}(jt_n^{2/3}) -\sigma t_n^{1/3}B'(j) \big\| \ge n^7t_n^{1/3} \Big) \le e^{-cn^3}.
\end{equation}
Rewriting the last equation with the Brownian motion $B(s):=t_n^{1/3}B'(t_n^{-2/3}s)$ we have 
\begin{equation}\label{eq:emb}
    \mathbb P \Big( \max _{j\le k} \big\| \tilde{W}(jt_n^{2/3}) -\sigma B(jt_n^{2/3}) \big\| \ge n^7t_n^{1/3} \Big) \le  e^{-cn^3}.
\end{equation}
Next, define the event 
\begin{equation}
    \mathcal C :=\Big\{ \text{For all } j \le k \text{ and }  s\in \big[ (j-1)t_n^{2/3},(j-1)t_n^{2/3} \big] \text{ we have }  \big\| B(s)-B(jt_n^{2/3}) \big\| \le n^2t_n^{1/3}  \Big\}
\end{equation}
and note that $\mathbb P (\mathcal C) \ge 1-e^{-cn^3}$. On $\mathcal C\cap \mathcal A$ and the complement of the event in the \eqref{eq:emb} we have that $\|\tilde{W}(s)-\sigma B(s)\| \le 2n^7t_n^{1/3}$ for all $s\le t_n$. This finishes the proof of the lemma using \eqref{eq:conca3}.
\end{proof}

\subsection{Transition Probabilities}

In this section we prove the inductive hypothesis on transition probabilities. To this end we use the coupling with Brownian motion at time $t_n$ from Section~\ref{sec:KMT} and the bound on transition probabilities at time $t_n^{4/5}$.

\begin{lem}
For all $t_n/4\le t \le t_n$ and $u\in \mathbb Z ^d$ we have that
\begin{equation}
    \mathbb P (W(t)=u) \le t^{-d/2+0.1}.
\end{equation}
\end{lem}

\begin{proof}
Let $t_n/4\le t \le t_n$ and $u\in \mathbb Z ^d$. Let $W_1$ and $W_2$ be independent regenerated walks  of lengths $t-t^{4/5}$ and $t^{4/5}$ respectively. Let $\tilde{W}$ be a concatenation of $W_1$ and $W_2$. By Lemma~\ref{lem:conca} there is a coupling such that $\|\tilde{W}(t)-W(t)\|\le t_n^{2\epsilon }$ with probability at least $1-e^{cn^3}$. Let $\mathcal B _1$ be the event that the coupling of $W_1$ with its Brownian motion $B$ was successful. We have that $\mathbb P (\mathcal B ) \ge 1- e^{-cn^2}$. Let $\mathcal A _2$ be the event that $\|W_2(t^{4/5})\| \le n^2 t^{2/5} $ and note that $\mathbb P (\mathcal A _2 )\ge 1-e^{-cn^2}$. For all $v\in \mathbb Z ^d$ we have
\begin{equation}
\begin{split}
    \mathbb P \big(\tilde{W}(t)=v\big) &\le \mathbb P ( \mathcal A _2 ^c )+ \sum _{ \substack{w: \\ \|w-v\| \le  n^2 t^{2/5} }    } \mathbb P \big(  W_1(t-t^{4/5})=w \big) \cdot \mathbb P  \big( W_2(t^{4/5}) =v-w \big) \\
    &\le e^{-cn^2} +t^{4/5(-d/2+0.1 )} \mathbb P \big( \| W_1(t-t^{4/5})-v \| \le n^2 t^{2/5} \big) \\
    &\le e^{-cn^2} +\mathbb P (\mathcal B _ 1 ^c) +t^{-2d/5+0.08 } \mathbb P \big( \| \sigma B(t-t^{4/5})-v \| \le 2n^2 t^{2/5} \big)\\
    &\le e^{-cn^2} +\mathbb P (\mathcal B _ 1 ^c) +C t^{-2d/5+0.08} n^{2d} t_n^{2d/5-d/2} \\
    &\le e^{-cn^2} +Cn^{2d}t^{-d/2+0.08} \le t^{-d/2+0.09}.
\end{split}
\end{equation}
Where in the second inequality we used the inductive bound on transition probabilities at time $t^{4/5}$. Thus, we obtain
\begin{equation}
    \mathbb P (W(t)=u) \le e^{-cn^3} +\mathbb P \big( \| \tilde{W}(t)-u \| \le t_n^{2\epsilon } \big) \le Ct_n^{2\epsilon d} t^{-d/2+0.09} \le t^{-d/2+0.1},
\end{equation}
where in the last inequality we used that $\epsilon = 1/(300d)$.
\end{proof}

\subsection{The Pair Proximity Property}
In this section we prove the pair proximity property in the inductive hypothesis. The proof uses heavily the arguments and notations of Section~\ref{sec:heavy}. Note that by translation invariance in space and time, we may assume without loss of generality that $u_1=0, q_1=0$. Thus, it suffices to prove the following lemma.

\begin{lem}\label{lem:pair induction step}
For all $u\in \mathbb Z ^d$ and $q\le \beta $ we have
\[
\mathbb P \Big(\Omega \text{ and }\big| \fP_n(0,u,0,q) \big| \ge 3.9^n \Big) \le e^{-cn^3} .
\]
\end{lem}

\begin{proof}
Recall the definition of $\hat{\tau }$ in equation~\eqref{eq:def of hat tau} and that $2^k$ is the scale of the bad blocks where $k:= \lfloor 4\delta n \rfloor $. We write $\hat{\tau }^{(1)}$ and $\hat{\tau }^{(2)}$ to denote the stopping times analogous to $\hat{\tau }$ for the walks $W_1$ and $W_2$ respectively. Let $\cH_1$ be the event that $\hat{\tau }^{(1)}> t_n'$ and that 
\begin{equation}
    \big\{ \forall s_1,s_2 \le t_n' \text{ with } |s_1-s_2|\le 3.6^n \text{ we have } \ \|W_1(s_1)-W_1(s_2)\| \le 1.9^n \big\}.
\end{equation}
Using the results of Section~\ref{sec:heavy}, the Brownian approximation in Lemma~\ref{cor:kmt} and a straight forward Brownian motion estimate (using that $1.9^2>3.6$) we obtain that $\mathbb P (\mathcal H _1 ^c \cap \Omega  )\le e^{-cn^3}$. Let $V=\{W_1(s):s\in[0,t_n']\}$ denote the vertices visited by $W_1$ in $[0,t_n']$. We will first reveal the path $W_1$ and then let $\cF_t^+$ denote the $\sigma$-algebra generated by the walk $W_2$ up to time $t$ together with the walk $W_1$ in $[0,t_n']$. We let $\mathcal T _2$ be the set of relaxed times with respect to $W_2$ and define the stopping time 
\begin{equation}
    \tau _{\text{merge}}:=  \inf \big\{ t_2>0 : \exists t_1\le t_n', \ t_2+q-t_1\in \beta \mathbb Z \text{ and } W_1(t_1)=W_2(t_2) \big\}.
\end{equation}
Note that on $\Omega$ we have that $\tau _{\text{merge}}>t_n'$. Next, let us define the following series of stopping times.  Let $\zeta_0=\zeta'_0=0$ and let
\begin{equation}
\zeta_i  :=t_n' \wedge \tau _{\text{reg}}^{(2)} \wedge \tau _{\text{merge}} \wedge \inf\big\{t\geq \zeta_{i-1}': \  t\in \mathcal T _2  \text{ and } d\big( W_2(t),V \big) \geq 1.9^k \big\},
\end{equation}
where $\tau _{\text{reg}}^{(2)}$ is defined analogously to \eqref{eq:def of tau reg} for the walk $W_2$. Next, define
\begin{align*}
\zeta_i'= t_n' \wedge \tau _{\text{hit}} ^{(2)}(\zeta_i) &\wedge \inf \big\{t\geq \zeta_i: d(V,W_2(t)) \leq 1\big\}\wedge \inf \big\{ t\geq \zeta_i + 3.89^n: d\big( W_2(t),V \big) \leq 1.9^n \big\},
\end{align*}
where $\tau _{\text{hit}} ^{(2)}$ is defined analogously to \eqref{eq:hit} with respect to $W_2$.  To count the number of such stopping times we let
\[
J=\inf \big\{ i:\zeta_i= t_n' \wedge \tau _{\text{merge}} \wedge \tau _{\text{reg}}^{(2)} \big\}.
\]

First we show that there is a large probability that $\zeta_i'-\zeta_i\geq 4^n$. To this end, let $\tilde{W}$ be a regenerated walk starting from $0$ and independent $\mathcal F _{\zeta _i}^+$. For $s\ge \zeta _i$ let $W'(s):=W_2(\zeta _i) +\tilde{W}(s-\zeta _i)$. By the same arguments as in Lemma~\ref{l:t.hit.couple}, conditional on $\cF_{\zeta_i}^+$, we can couple $\{W_2(s)\}_{s\geq \zeta_i}$ with $\{W'(s)\}_{s\geq \zeta_i}$ until $\zeta _i'$.

First, using Lemma~\ref{lem:relaxed4} and the fact that $\zeta _i\in \mathcal T _2$ we have that 
\begin{equation}
 \mathbb P \big( \zeta _i' =\tau _{\text{hit}} ^{(2)} (\zeta _i)  \mid \mathcal F _{\zeta _i}^+ \big) \le  \beta ^{-c}.
\end{equation}

Next, note that on $\mathcal H _1$ we have $|N(V,1) \cap [-r,r]^d| \le Cr^{5/2}$ for all $r\ge 1.9^k$. Indeed, by Lemma~\ref{lem:heavy2}, before $\hat{\tau }^{(1)}$ non of the blocks of side length at least $t_n^ \delta <\!< 1.9^k$ are heavy with respect to $W_1$. Thus, by Claim~\ref{claim:A} and using that $d\big( W'(\zeta _i) ,V \big) \ge 1.9
^k$ we obtain 
\begin{equation}
    \mathbb P \big( \forall \zeta _i \le s\le \zeta _i +4^n ,\ d(W'(s),V) >1  \mid \mathcal F _{\zeta _i}^+ \big) \ge 1-e^{-ck}.
\end{equation}

% Next, we have that 

% Since $W_2(\zeta_i)$ is relaxed$^*$ and $d(W_2(\zeta_i),\Lambda^{(2)}_t) \geq \frac12 Q^3 2^{k}$, the calculation in Lemma~\ref{l:relaxed*} implies that on the event $ \{ \tau_{{\rm bad}}^{(2)} \geq \zeta_i \}$ we have 
% \begin{equation}
% \P\Big(\inf \big\{ t\geq \zeta_i: \exists s\in[q,\zeta_i), \ |W_2(s)-W'(t)|\leq 1, s\equiv t \hbox{ mod } \beta \big\} \geq  \zeta_i + 4^n  \ \big| \ \cF_{\zeta_i}^+\Big) \geq 0.9.    
% \end{equation}

% Next we consider hitting $V$.  Let $\ell_0=\lfloor \log _2 (1.9^k)\rfloor$.  On the event $\cH_1$ we have that no boxes of size $2^k$ or higher are heavy$^*$.  Moreover, all of the bad blocks are distance at least $\frac12 Q^3 2^{k}$ away.
% Hence for $\ell\geq \ell_0$,
% \[
% \big|\{u\in V\setminus \Lambda^{(1)}:|u-W_2(\zeta_i)|\leq 2^\ell\}\big|\leq 2^d 5^{k\vee\ell}<6^\ell,
% \]
% while for $\ell<\ell_0$
% \[
% \big|\{u\in V\setminus \Lambda^{(1)}):|u-W_2(\zeta_i)|\leq 2^\ell\}\big|=0.
% \]
% Hence by Claim~\ref{claim:A}, on the event $\cH_1$ we have 
% \[
% \P\Big(\inf \big\{ s\geq \zeta_i: d\big( V \setminus \Lambda^{(1)},W'(s)\big) \leq 1 \big\} \geq \zeta_i+4^n  \mid \cF_{\zeta _i}^+\Big) \geq 0.9.
% \]
% By an essentially identical proof to Lemma~\ref{l:relaxed*} and repeating the bounds on equations 
% \eqref{eq:A1} and \eqref{eq:A2} we obtain that on $\cH_1$ we have 
% \[
% \P\Big( \inf\big\{ s\geq \zeta_i: d\left( \Lambda^{(1)},W'(s)\right)\leq 1 \big\} \geq \zeta_i+4^n \mid \cF_{\zeta _i}^+\Big) \geq 0.9.
% \]

Finally we need to check that after time $\zeta _i+3.89^n$ the walk $W'(t)$ is at distance at least $1.9^n$ away from $V$ with high probability. On the event $ \mathcal H_1 $ we have that 
\begin{equation}
 N(V, 1.9^{n}) \subseteq  \bigcup _{i=0}^{t_n'/3.6^n}  B\big( W_1(i\cdot 3.6^n), 1.9^{n+2} \big).
\end{equation}
Thus, on the event $\mathcal H _1 $ by the Brownian approximation and Lemma~\ref{l:BM.ball.exit} we have
\begin{equation}
\begin{split}
    \mathbb P \Big( &\exists t \in [\zeta _i+3.89^n,\zeta _i +4^n], \ W'(t)\in N(V, 1.9^{n}) \ \big| \ \mathcal F ^+_{\zeta _i} \Big) \\
    &\le \sum _{i=0}^{t_n'/3.6^n} \mathbb P \Big( \exists t \in [\zeta _i+3.89^n,\zeta _i +4^n], \ \big\| W'(t)- W_1(i\cdot 3.6^n) \big\| \le  1.9^{n+2}  \ \big| \ \mathcal F ^+_{\zeta _i} \Big)\\
    &\quad \quad \quad \quad \quad \quad \quad \quad \quad \quad \le C(4/3.6)^n \big( 1.9/\sqrt{3.89} \big) ^{(d-2)n} \le e^{-cn},
\end{split}
\end{equation}
where in the last inequality we used that $d\ge 5$.

Combining the above estimates we have that on the event $\mathcal H _1 $ 
\[
\P\big( \zeta _i' \geq \zeta_i+4^n    \mid \cF_ {\zeta _i}^+\big) \geq 1/2.
\]
So each excursion starting from $\zeta_i$ has probability at least $1/2$ of lasting for at least time $4^n$.  We can have at most $1$ such intervals before time $t_n'$ and so
\begin{equation}\label{eq:J is small}
\P\big( J\geq n^3 \big)  \leq \P \big( \hbox{Bin}(n^3,1/2) \leq 2 \big) \leq e^{-cn^3}.   
\end{equation}

Between $3.89^n + \zeta_i$ and $\zeta_i'$ the distance from $W_2(t)$ to $V$ is at least $1.9^n$ and therefore
\begin{equation}\label{eq:fn.bound}
\big| \fP_n(0,u,0,q) \big| \leq \sum_{i=1}^{J} 3.89^n + \zeta_i  - \zeta_{i-1}'.
\end{equation}
It suffices to bound $\zeta _i-\zeta _{i-1}'$ with high probability. For $t\le t_n'$ let $\Lambda '_t$ be the set of bad blocks discovered by both $W_1$ in the interval $[0,t_1]$ and $W_2$ in the interval $[0,t]$ (note that because of the pair property, these bad blocks might contain blocks that are good with respect to $(W_1(s) : s\le t_n')$ and good with respect to $(W_2(s):s\le t)$). We define the stopping time $\tau ' _{\text{bad}}$ with respect to the filtration $\mathcal F _{t}^+$ analogously to \eqref{eq:def of tau bad} with $\Lambda '_t$. By the same arguments we have that $\mathbb P (\tau ' _{\text{bad}}\le t_n') \le e^{-cn^3}$.

Let $t_k'\le t<\hat{\tau } ^{(2)}$ such that $W_2(t)\notin \Lambda '_t$. By Lemma~\ref{l:relaxed.pts} we have that
\begin{equation}\label{eq:T2}
   \big| \mathcal T _2 \cap [t-t_k',t] \big| \ge 0.9t_k'.
\end{equation}
Moreover, using the pair property just like in the proof of Lemma~\ref{l:relaxed.pts}, on the event $\mathcal H _1 \cap \{t<\hat{\tau }^{(2)}, W_2(t) \notin \Lambda '_t\}$ we have
\begin{equation}\label{eq:W_1 W_2 far}
    \big| \big\{  s_2\in [t-t_k',t] : \forall s_1\le t_n', \ \|W_1(s_1)-W_2(s_2)\| > 1.9^k \big\}  \big| \ge 0.9 t_k'.
\end{equation}

Next, just like in the proof of Corollary~\ref{cor:relaxed.pts}, if $ \tau '_{\text{bad}}>t_n'$ then 
\begin{equation}
  \big|  \big\{ t\le t_n' : W_2(t)\in \Lambda ' _t\big\} \big| \le t_n^{\epsilon/3}
\end{equation}
in both cases, when $t_n^{5\delta }\le \beta $ and when $t_n^{5\delta }\ge \beta $. Thus, by \eqref{eq:T2} and \eqref{eq:W_1 W_2 far}, on the event $\mathcal H:=\mathcal H _1 \cap \{ \hat{\tau }^{(2)} \wedge \tau '_{\text{bad}} >t_n'\}$ we have that any interval $I\subseteq [0,t_n']$ of length $t_n^{\epsilon }$ contains a time $s_2\in \mathcal T _2$ such that $\|W_2(s_2)-W_1(s_1)\|\ge 1.9^k$ for all $s_1\le t_n'$. It follows that on $\mathcal H$ we have that $\zeta _{i}-\zeta _{i-1}'\le t_n^{\epsilon }$ for all $i\ge 1$ and therefore on $\mathcal H \cap \{J\le n^3\}$ we have 
\[
\big| \fP_n(0,u,0,q) \big| \le \sum_{i=1}^{J} 3.89^n + \zeta_i  - \zeta_{i-1}' \le J ( 3.89^n+t_n^{\epsilon }) \le 3.9^n .
\]
This finishes the proof of the lemma by \eqref{eq:J is small} and the fact that $\mathbb P (\mathcal H ^c \cap \Omega)\le e^{-cn^3}$.
\end{proof}

\section{Induction base}\label{sec:base}

In this section we prove the inductive assumption when $t_n=t_0 4^n \le \beta$. For time less than $\beta$, the cyclic walk is simply a continuous time simple random walk and thus the \emph{transition probabilities} properties holds by standard properties of random walks.  Similarly for the \emph{travelling far} property we have
\begin{equation}\label{eq:induct.base.far}
\mathbb P \Big( \sup _{s\le t_n} \|W(s)\| \leq n^2\sqrt{t_n} \Big) \geq 1-e^{-cn^4}.
\end{equation}
By Claim~\ref{claim:fast}
\begin{equation}\label{eq:induct.base.fast}
\P\Big(\sup_{t,t'\leq t_n,|t-t'|\leq 1} |W(t)-W(t')|\leq n^5 \Big) \geq 1- e^{-n^4}
\end{equation}
and so by KMT embedding from Theorem~\ref{thm:KMT} we have a coupling with Brownian motion such that
    \begin{equation}\label{eq:induct.base.brown.approx}
        \mathbb P \Big( \max _{t\le t_n} | W(t)-  B(t)| \leq n^{10} \Big) \geq 1- e^{-n^4}
    \end{equation}
establishing the \emph{approximately Brownian} property. It remains to verify the relaxed path property and the pair property.

\begin{lem}
Suppose that $t_n\le \beta $. Then, the path up to time $t_n$ is relaxed with probability at least $1-e^{-n^3}$.
\end{lem}

\begin{proof}
Let $\cA_1$ be the intersection of the events in equations~\eqref{eq:induct.base.far},~\eqref{eq:induct.base.fast} and~\eqref{eq:induct.base.brown.approx}.  So $\P(\cA _1)\ge 1-e^{-cn^4}$.
Let $\cA_2$ be the event that for $B(t)$ all boxes of scale $\beta^{\delta}$ or higher are very lite,
\[
\cA_2:=\Big\{\forall r\ge \beta ^\delta ,x\in \mathbb Z ^d ,  \  \big| \{i\in [0,t_n]\cap \Z: B(i)\in B(x,r) \}\big|  < r^{2.1} \Big\}.
\]
Note that when $r\ge t_n$ we have $ |\{i\in [0,t_n]\cap \Z: B(i)\in B(x,r) \}| < r^{2.1}$ for all $x\in \mathbb Z ^d$ and that on the event $\cA_1$ only blocks within distance $n^2 2^n$ are visited. Moreover, for any $r\ge \beta ^\delta $ and $x\in \mathbb Z ^d$, by Lemma~\ref{l:heavy.BM.ball}
\[
\P \big(  \big| \{i\in [0,t_n]\cap \Z: B(i)\in B(x,r) \} \big| < r^{2.1} \big) \geq 1- e^{-n^4}.
\]
Thus, taking a union bound over $\|x\|\le 2n^22^n$ and $\beta ^\delta \le r \le t_n$ we obtain $
\P( \cA_2 ) \ge 1-e^{-cn^4}$.
On the event $\cA_1\cap\cA_2$ we claim that all boxes $B(x,r)$ for $r\ge \beta ^\delta $ are not heavy with respect to $W$. Indeed, let $x\in \mathbb Z ^d$ and $r\ge \beta ^\delta$ and suppose that $W(t)=u \in B(x,r)$ for some $t\leq t_n$.  Then by~\eqref{eq:induct.base.fast} and the coupling we have that $|B(\lfloor t \rfloor)-u| \leq 2n^{10}$. We can cover every vertex within distance $2n^{10}\le r$ of $B(x,r)$ by $3^d$ disjoint blocks of the form $B(y,r)$ with $y\in \mathbb Z ^d$. Thus, on $\mathcal A _1 \cap \mathcal A _2$
\begin{equation}\label{eq:induct.base.no.heavy}
\big| \{W(t):t\leq t_n\}\cap B(x,r) \big| \leq 3^d (4n^{10})^d r^{2.1} < r^{5/2}.
\end{equation}
Now let
\[
T'=\big\{  t\leq t_n: \forall 0\le s \le t-\beta^{3\delta }, \   |B(s)-B(t)| \geq 2\beta^{\delta } \big\},
\]
and
\[
T=\big\{  t\leq t_n: \forall 0\le s \le t-\beta^{3\delta }, \  |W(s)-W(t)| \geq \beta^{\delta } \big\}.
\]
On the event $\cA_1$, using that $n^{10}<<\beta ^{\delta }$, we have $T'\subset T$.  Furthermore, on $\mathcal A _1 \cap \mathcal A _2$, using equation~\eqref{eq:induct.base.no.heavy},  we have that all times $t\in T$ are relaxed. Thus, it suffices to show that $T'$ is large with high probability.  If $t_n \leq \beta^{3\delta }$ then $|T'|=|T|=t_n$.  Suppose next that  $\beta^{3\delta }\leq t_n \leq \beta$. Set
\[
B^*(t)=\frac{1}{2\beta^{\delta } } B\big(4\beta^{2\delta }t\big).
\]
which is also distributed as Brownian motion.  Then letting 
\[
T^*=\Big\{t\leq \frac{t_n}{4\beta ^{2\delta }}   : \forall s \leq t - \beta^{\delta }/4 \hbox{ with } |B^*(s)-B^*(t)| \ge 1\Big\}
\]
we have that $T'=4\beta^{2\delta }T^*$. Thus, by Claim~\ref{claim:brownian}  with $\zeta=\delta /2$ we obtain
\begin{equation}
\P \big( |T'| \geq 0.95 t_n \big) =\P \Big( |T^*| \geq 0.95 \frac{t_n}{4\beta ^{2\delta} }\Big) \geq 1- \exp\Big(-\big( \frac{t_n}{4\beta ^{2\delta }} \big)^c\Big) \geq 1- e^{-n^4}.
\end{equation}
On the event $\{|T| \geq 0.95t_n\}\cap \cA_1\cap\cA_2$ at least a $0.95$ fraction of the times are relaxed and so the path is relaxed.  This has probability at least $1-e^{-cn^4}$ which completes the proof.
\end{proof}

Finally we complete the base of the induction by proving the pair proximity property. As before, we may assume that $u_1=0$ and that $q_1=0$. Note also that by reversibility in time we may assume that $q_2\le \beta /2$.

\begin{lem}
For any $u\in \mathbb Z ^d$ and $q\le \beta /2 $ we have that
    \begin{equation}\label{eq:path.pair.base}
       \mathbb P \Big(\Omega  \text{ and }\big| \fP_n (0,u,0,q) \big| \ge 3.9^n \Big) \le e^{-n^3} .
    \end{equation}
\end{lem}
\begin{proof}
Recall that $W_1$ and $W_2$ are two cyclic walks of length $t_n'=t_n/n^4\le \beta /2$, starting from $0$ and $u$ at the cyclic times $0$ and $q$.

We will begin with the case that $q\leq t_n'$.  We let $Y_v$ be the cyclic walk started from $v\in \Z^d$ at time $0$.  Then $W_1$ corresponds to $Y_0$ while $W_2$ must correspond to $Y_v$ for some $v\in \Z^d$. Let
\[
\fP^*(v) =\{s_2\leq 2 t_n' : \exists s_1 \in[0, 2t_n'], \ |Y_v(s_2) - Y_0(s_1)|\leq 1.9^n\}.
\]
Recall that $Y_v$ are simple random walks in this time scale. On the event $\Omega$ we have
\[
\big| \fP_n(0,u,0,q) \big| \leq \max_{v\neq 0} \big|\fP^*(v)\big|.
\]

For all $v\in \mathbb Z ^d$ with $\|v\|\ge 4^n$, we have that 
\begin{equation}
    \mathbb P (\fP^*(v) \neq  \emptyset ) \le \mathbb P \Big( \sup _{s\le 2t_n'} \|Y_0(s)\| \ge \|v\|/3 \Big) +\mathbb P \Big( \sup _{s\le 2t_n'} \|Y_v(s)-v\| \ge \|v\|/3 \Big) \le \exp (-c\|v\|),
\end{equation}
where the last inequality follows from standard random walk estimates. Thus, the event
\begin{equation}
    \mathcal A := \big\{ \forall v\in \mathbb Z ^d \text{ with } \|v\|\ge 4^n \text{ we have } \fP^*(v) =  \emptyset\big\}
\end{equation}
satisfies $\mathbb P (\mathcal A ) \ge 1-\exp (-c4^n)$.

Next, fix $v$ such that $\|v\|\le 4^{n}$. Consider the evolution of $Y_0(s)-Y_v(s)$ for $s\in[0,2t_n']$.  It is a Markov process with jump rate identical to a continuous time random walk with jump rate doubled except when $\|Y_0(s)-Y_v(s)\|=1$.  Since simple random walk is transient, in each unit interval of time there is a positive probability $\|Y_0(s)-Y_v(s)\|$ becomes at least 2 and then never returns to 1.  Hence
\[
\P\Big( |\{i\in[0,2t_n]\cap\Z: \min_{s\in[i,i+1)} \|Y_0(s)-Y_v(s)\|=1 \}| \geq n^4 \Big) \leq e^{-cn^4}.
\]
We can couple the two walks to independent random walks outside of these intervals.
In particular, we can couple $Y_v(s)$ with a simple random walks $X(s)$ which is independent of $Y_0(s)$ such that the event
\[
\cD=\Big\{ \sup_{s\in[0,2t_n]} \|Y_v(s) - X(s)\| \leq n^{10} \Big\}
\]
holds with very high probability, $\P(\cD) \geq 1- e^{-cn^4}$. 

We now follow the proof of Lemma~\ref{lem:pair induction step}. Let $V:=\big\{ Y_0(s) : s\in [0,2t_n']  \big\}$ be the vertices visited by the first walk. Let $\mathcal F _t ^+$ be the sigma algebra generated by the first walk up to time $2t_n'$ and the second walk up to time $t$. Define the following sequence of stopping times. Let $\zeta _i=0$ and 
\begin{equation}
    \zeta _i: = \inf \big\{ t\ge \zeta _i +3.89, \  d( X(t), V) \le 1.9^{n+1}\big\}.
\end{equation}
Next, define the event
\begin{equation}
\cC=\Big\{\sup_{t\in [0,2t_n']}\sup_{s\in[t-(3.6)^n,t]}\|Y_0(t)-Y_0(s)\| \leq 1.9^n\Big\}\\
\end{equation}
and note that by standard random walk estimates (using that $1.9^2> 3.6$) we have that $\P(\cC) \geq 1-e^{-n^4}$.

Now, by the same arguments as in the proof of Lemma~\ref{lem:pair induction step}, on the event $\mathcal C$ we have that
\begin{equation}
    \mathbb P \big(  \zeta _i\ge 2t_n' \mid \mathcal F _{\zeta _{i-1}}^+ \big) \ge 1/2.
\end{equation}
Thus, letting $J:=\inf \big\{ i : \zeta _i \ge 2t_n' \big\}$ we have $ \mathbb P \big( J\ge n^4 \big) \le e^{-cn^4}$. On the complement of this event and the event $\mathcal D$ we can bound
\begin{equation}
    |\mathcal B ^*(v)| \le J\cdot 3.89^n \le 3.9^n
\end{equation}

Taking a union bound over $v$ we have that
\[
\P \big(\big| \fP_n (0,u,0,q) \big|\geq 3.9^n \big) \leq \P(\cA) + \sum_{\|v\|\leq 4^n}\P \big( \big|\fP^*(v)\big|\geq 3.9^n \big) \leq e^{-cn^4}
\]
which completes the proof when $q\le t_n'$. If $t_n'<q<\beta /2$ then the walks $W_1$ and $W_2$ are independent and the proof is identical.
\end{proof}

\section{Brownian Estimates}\label{sec:brownian}

The following is a standard application of the Optional Stopping Theorem  and the fact that $|B(t)|^{2-d}$ is a local martingale.
\begin{lem}\label{l:BM.ball.hit}
Let $D$ be a ball of radius $r$ centred at $x$ and let $B(t)$ be Brownian motion started from $y$ such that $|v-u| >r$.  Then the probability that $B(t)$ ever hits $D$ is $\big( \frac{r}{|x-y|} \big) ^{d-2}$.
\end{lem}

\begin{lem}\label{l:BM.ball.exit}
Let $B(t)$ be Brownian motion in $\R^d$.  Then there exists $C_d>0$ such for $0<r\leq s$ and any starting point $y\in \R^d$,
\[
\P\Big( \inf_{t\geq s^2} \|y+B(t)\| \leq r \Big) \leq C_d (r/s)^{d-2}.
\]
\end{lem}

\begin{proof}
By Lemma~\ref{l:BM.ball.hit} we have that
\begin{align*}
\P\Big( \inf_{t\geq s^2} \|B(t)\| \leq r \Big) &= \E\bigg[\big(\frac{r}{\|y+B(s^2)\|}\big)^{d-2} \wedge 1 \bigg]\leq \E\bigg[\big(\frac{r}{\|B(s^2)\|}\big)^{d-2} \wedge 1 \bigg]\\
&= \E\bigg[\big(\frac{r/s}{\|B(1)\|}\big)^{d-2} \wedge 1 \bigg]= C\int_{0}^\infty x^{d-1} (r/xs\wedge 1)^{d-2} e^{-x^2/2} dx\\
&\leq C \int_0^{r/s} x^{d-1} e^{-x^2/2} dx + C (r/s)^{d-2}\int_{r/s}^\infty x  e^{-x^2/2} dx \leq C_d (r/s)^{d-2},
\end{align*}
where the first inequality follows from expectation being maximized at $y=0$, the second inequality by Brownian scaling and the third equality follows by the radial symmetry of the Gaussian density.
\end{proof}

\begin{lem}\label{l:heavy.BM.ball}
Let $D$ be a ball of radius $r$ and let $N$ be the total number of integer times Brownian motion is inside $D$.  Then for some $c>0$,
\[
\P( N > x r^2) \leq 2\exp(-cx).
\]
\end{lem}
\begin{proof}
For any starting point, by Lemma~\ref{l:BM.ball.exit} there is a constant probability, not depending on $r$, that the Brownian motion never returns to $D$ after time $Cr^2$. The result then follows by repeated trials.
\end{proof}

\begin{claim}\label{claim:brownian} 
Let $d\ge 5$ and let $B$ be a standard $d$ dimensional Brownian motion. Let $\zeta >0$ and let $t\ge 1$. Define the set 
\begin{equation}
    T:=\big\{ s\le t : \exists s'\le s-t^{\zeta } \text{ with } \|B(s)-B(s')\| \le 1 \big\}.
\end{equation}
We have that $\mathbb P \big( |T|\ge t^{1-\zeta^3} \big) \le C\exp  (-t^{c})$ where $c$ depends on $\zeta $.
\end{claim}

\begin{proof}
Without loss of generality, we may assume that $d=5$. Indeed, the set $T$ grows when removing the last $d-5$ coordinates. Let $r=t^{\zeta /2}$. First, we claim that for all $x\in \mathbb R^5$ we have
\begin{equation}\label{eq:15}
    \mathbb P \big( \big|\big\{ i \in \mathbb N \cap [0,t] : \|B(i)-x\| \le 2r^{1+\zeta } \big\}\big| \ge r^{2+3\zeta  } \big) \le \exp (-Ct^{\zeta}).
\end{equation}
Indeed, every time the Brownian motion is inside a ball of radius $2dr^{1+\zeta }$ around $x$ by Lemma~\ref{l:BM.ball.exit} there probability $\theta>0$ to escape the ball in time less than $r^{2+2\zeta}$ and never return. The estimate then follows from repeated trials.

Next, we say that the block $A_x:=x+[-r^{1+\zeta  },r^{1+\zeta }]^5 \subseteq \mathbb R^5$ is  heavy at time $s$ if 
\begin{equation}
    \big| \big\{ y\in A_x \cap \Z^d: \exists s'\le s \text{ with }\|B(s')-y\| \le r^{1/20 } \big\} \big| \ge  r^{2.1}
\end{equation}
and define the stopping time
\begin{equation}
    \tau _{\text{hea}}:= \inf \big\{ s>0 : \text{ for some }x\in \mathbb Z ^5 \text{ the block } A_x \text{ is heavy at time }s  \big\}.
\end{equation}
We would like to show that it is unlikely that $\tau _{\text{hea}} \le t$. To this end, note that 
\begin{equation}\label{eq:16}
    \mathbb P \big( \exists s,s'\le t \text{ with } |s-s'|\le 1 \text{ such that } \|B(s)-B(s')\| \ge r^{1/20} \big) 
    \le C\exp (-c t^\zeta).
\end{equation}
Moreover, if both the events in \eqref{eq:15} and \eqref{eq:16} do not hold then 
\begin{equation}
    \big| \big\{ y\in A_x \cap \Z^d : \exists s\le t \text{ with }\|B(s)-y\| \le r^{\zeta } \big\} \big| \le r^{2+3\zeta } (2r^{1/20 })^5  <   r^{2.1}
\end{equation}
and therefore $\mathbb P (A_x \text{ is heavy at time }t) \le C\exp (-t^c)$. We obtain that 
\begin{equation}
    \mathbb P (\tau _{\text{hea}} \le t ) \le \mathbb P \big( \sup _{s\le t} \|B(s)\| \ge t \big) +  \sum _{\|x\|\le 2t} \mathbb P \big( A_x \text{ is heavy at time }t \big)  \le C\exp (-t^c).
\end{equation}

Next, for any $j\le \lfloor t^{1-\zeta } \rfloor$ let $s_j:=jt^{\zeta }=jr^2$ and define the event
\begin{equation}
    \mathcal A _j : =\big\{  \exists s \in [s_j+r^{2-2\zeta },s_{j+1}] \text{ and }s'\le s_j \text{ with } \|B(s)-B(s')\|\le 1  \big\}.
\end{equation}
We claim that on the event $\{ \tau _{\text{hea}} \ge  s_j \}$ we have that $ \mathbb P \big( \mathcal A _j   \ |  \  \mathcal F _{s_j} \big) \le r^{-1/2}$.
Indeed, for any fixed $s\in [s_j+r^{2-2\zeta },s_{j+1}]$ the density of $B(s)-B(s_j)$ is bounded by $Cr^{-5(1-\zeta )}$ and therefore 
\begin{equation}\label{eq:17}
    \mathbb P \big( \exists s'\le s_j \text{ with } \|B(s)-B(s')\|\le r^{\zeta }   \ |  \  \mathcal F _{s_j} \big) \le Cr^{2.1} r^{-5(1-\zeta ) } \le r^{-2.8},
\end{equation}
where in here we also used that the block $A_x$ is not heavy where $x\in \mathbb Z ^d$ is the closest integer point to $W(s_j)$ and that $W(s)\in A_x$ with very high probability. We now union bound \eqref{eq:17} over integers $s\in [s_j+r^{2-2\zeta },s_{j+1}]$ and use \eqref{eq:16} to obtain  $ \mathbb P ( \mathcal A _j \mid  \mathcal F _{s_j} ) \le r^{-1/2}$.

It follows that 
\begin{equation}
    \mathbb P \big( \big| \big\{ j\le \lfloor t^{1-\zeta } \rfloor : \mathcal A _j \text{ holds} \big\} \big|  \ge t^{1-1.1\zeta}  \big) \le C\exp (-t^c).
\end{equation}
On the complement of the last event we have that
\begin{equation}
    |T|\le r^{2-2\zeta } \big| \big\{ j\le \lfloor t^{1-\zeta } \rfloor : \mathcal A _j \text{ doesn't hold} \big\} \big| + r^2\big| \big\{ j\le \lfloor t^{1-\zeta } \rfloor : \mathcal A _j \text{ holds} \big\} \big| \le t^{1-\zeta ^3}.
\end{equation}
This finishes the proof of the claim.
\end{proof}

\bibliography{Interchange}
\bibliographystyle{abbrv}

\end{document}